\documentclass[11pt]{amsart}
\usepackage{color}
\usepackage{amsmath}
\usepackage{amsxtra}
\usepackage{amscd}
\usepackage{amsthm}
\usepackage{amsfonts}
\usepackage{amssymb}
\usepackage{eucal}
\usepackage{epsfig}
\usepackage{graphics}
\usepackage[matrix,arrow]{xy}
\usepackage{hyperref}
\usepackage{bm}

\theoremstyle{plain}
\newtheorem{theorem}{Theorem}[section]
\newtheorem{lemma}[theorem]{Lemma}
\newtheorem{thm}[theorem]{Theorem}
\newtheorem{cor}[theorem]{Corollary}
\newtheorem{prop}[theorem]{Proposition}
\newtheorem{lem}[theorem]{Lemma}
\newtheorem{conj}[theorem]{Conjecture}
\newtheorem{defi}[theorem]{Definition}
\newtheorem{example}[theorem]{Example}
\newtheorem{rem}[theorem]{Remark}

\newcommand{\Glie}{\mathfrak{g}}
\newcommand{\Yim}{\mathcal{Y}}

\newcommand{\ZZ}{\mathbb{Z}}
\newcommand{\CC}{\mathbb{C}}

\newcommand{\C}{\mathbb{C}}
\newcommand{\Q}{\mathbb{Q}}
\newcommand{\Z}{\mathbb{Z}}

\newcommand{\g}{\mathfrak{g}}

\newcommand{\tb}{\mathbf{\mathfrak{t}}}
\newcommand{\ga}{\overline{\alpha}}

\newcommand{\Psib}{\mbox{\boldmath$\Psi$}}
\newcommand{\Psibs}{\scalebox{.7}{\boldmath$\Psi$}}
\newcommand{\qbin}[2]{{\left[
\begin{matrix}{\,\displaystyle #1\,}\\
{\,\displaystyle #2\,}\end{matrix}
\right]
}}

\newcommand{\nc}{\newcommand}
\nc{\on}{\operatorname}
\nc{\la}{\lambda}
\nc{\wh}{\widehat}
\nc{\wt}{\widetilde}
\nc{\sw}{{\mathfrak s}{\mathfrak l}}
\nc{\ghat}{\wh{\g}}
\nc{\hhat}{\wh{\h}}
\nc{\mc}{\mathcal}
\nc{\bi}{\bibitem}
\nc{\pa}{\partial}
\nc{\ppart}{(\!(t)\!)}
\nc{\pparl}{(\!(\la)\!)}
\nc{\zpart}{(\!(z^{-1})\!)}
\nc{\n}{{\mathfrak n}}
\nc{\ol}{\overline}
\nc{\mb}{\mathbf}
\nc{\bb}{{\mathfrak b}}
\nc{\su}{\wh\sw_2}
\nc{\h}{{\mathfrak h}}
\nc{\can}{\on{can}}
\nc{\ntil}{\wt{\n}}
\nc{\pone}{{\mathbb P}^1}
\nc{\bs}{\backslash}
\nc{\al}{\alpha}
\nc{\gt}{{\mathfrak g}'}
\nc{\ds}{\displaystyle}

\nc{\OO}{{\mc O}}

\newcommand{\mfr}{\mathfrak{r}}

\nc{\bchi}{{\boldsymbol \chi}}

\theoremstyle{definition}
\numberwithin{equation}{section}

\begin{document}

\begin{title}
{Extended Baxter relations and $QQ$-systems for quantum affine algebras}
\end{title}

\author{Edward Frenkel}

\address{Department of Mathematics, University of California,
  Berkeley, CA 94720, USA}

\author[David Hernandez]{David Hernandez}

\address{Universit\'e Paris Cit\'e and Sorbonne Universit\'e, CNRS,
  IMJ-PRG, F-75006, Paris, France}

\begin{abstract}
 Generalized Baxter's $TQ$-relations and the $QQ$-system are systems
 of algebraic relations in the category ${\mc O}$ of representations
 of the Borel subalgebra of the quantum affine algebra $U_q(\ghat)$,
 which we established in our earlier works \cite{FH,FH2}. In the
 present paper, we conjecture a family of analogous relations labeled
 by elements of the Weyl group $W$ of $\g$, so that the original
 relations correspond to the identity element. These relations are closely
 connected to the $W$-symmetry of $q$-characters established in
 \cite{FH3}. We prove these relations for all $w \in W$ if $\g$ has
 rank two, and we prove the extended $TQ$-relations if $w$ is a simple
 reflection. We also generalize our results and conjectures to the
 shifted quantum affine algebras.
\end{abstract}

\maketitle

\setcounter{tocdepth}{2}
\tableofcontents

\section{Introduction}

Let $\mathfrak{g}$ be a finite-dimensional simple Lie algebra over
$\C$ and $I$ the set of vertices of its Dynkin diagram.  Let
$U_q(\ghat)$ be the corresponding quantum affine algebra of level 0
and $U_q(\wh{\mathfrak{b}})$ its Borel subalgebra, defined in terms of
the Drinfeld-Jimbo generators of $U_q(\ghat)$. In \cite{HJ}, M. Jimbo
and one of the authors introduced the category $\OO$
  whose objects are representations of $U_q(\wh{\mathfrak{b}})$
(possibly infinite-dimensional) that have a weight space decomposition
with finite-dimensional weight components satisfying a certain
condition (see Definition \ref{defo}). It contains the category of
finite-dimensional representations of $U_q(\wh{\mathfrak{g}})$ (of
type 1) as a subcategory.

In this paper we study various algebraic relations in the (completed)
Grothendieck ring $K_0(\OO)$ of the category $\OO$, and the
corresponding relations on $q$-characters. They can be viewed as
extensions of the {\em generalized Baxter $TQ$-relations} and the {\em
  $QQ$-system}. Here's a brief summary of these relations and their
connection to the spectra of commuting quantum Hamiltonians of the
XXZ-type model associated to $U_q(\ghat)$.

  \medskip
  
(1) The $TQ$-relations are generalizations of the celebrated
Baxter $TQ$-relation of the quantum XXZ
model corresponding to the quantum affine algebra $U_q(\wh{\sw}_2)$:
\begin{equation}    \label{baxter}
T_a = u \frac{Q_{aq^2}}{Q_a} + u^{-1} \frac{Q_{aq^{-2}}}{Q_a}.
\end{equation}
It can be viewed (see e.g. \cite{FH}) as a relation in $K_0(\OO)$,
with $T_a = [L(Y_{aq^{-1}})]$, the class of the two-dimensional
representation of the Borel subalgebra of $U_q(\wh{\sw}_2)$ 
(the restriction of the fundamental representation of
  $U_q(\wh\sw_2)$); $Q_a = [L(\Psib_a)]$, the class of its {\em
  prefundamental representation}; and $u$ is the class of a
one-dimensional representation (these notions will be introduced
below).

There are analogues of this relation in the ring $K_0(\OO)$ for an
arbitrary Lie algebra $\g$ and any finite-dimensional irreducible
representation $V$ of $U_q(\wh\g)$, which were proved in
\cite{FH}. The starting point is the observation that the RHS of
formula \eqref{baxter} can be obtained by replacing the variable $Y_a$
appearing in the {\em $q$-character} of the two-dimensional
representation $L(Y_{aq^{-1}})$,
$$
\chi_q(L(Y_{aq^{-1}})) = Y_{aq^{-1}} + Y_{aq}^{-1},
$$
with the ratio $u^{-1} \dfrac{Q_{aq^{-1}}}{Q_{aq}}$.

Likewise, in general one should replace each variable $Y_{i,a}, i \in
I$, appearing in the $q$-character $\chi_q(V)$ of a finite-dimensional
representation $V$ of $U_q(\ghat)$ with the ratio $u_i^{-1}
\dfrac{Q_{i,aq_i^{-1}}}{Q_{i,aq_i}}$, where $Q_{i,a} =
      [L(\Psib_{i,a})]$, the class of the $i$th prefundamental
      representation in $K_0(\OO)$, $q_i = q^{d_i}$ with $d_i$ being
      the length of the $i$th simple root, and $u_i$ is the class of a
      one-dimensional representation. By equating the resulting
      expression with $[V]$, we obtain a relation in the field of
      fractions of $K_0(\OO)$ (and clearing the denominators, we
      obtain a polynomial relation in $K_0(\OO)$). This is the
      {\em generalized Baxter $TQ$-relation} associated to the
      representation $V$ of $U_q(\wh\g)$ proved in \cite{FH}
      (following a conjecture of N. Reshetikhin and one
        of the authors in \cite{Fre}).

\medskip

(2) Using the universal $R$-matrix of $U_q(\wh\g)$, for any
finite-dimensional representation $U$ of $U_q(\wh\g)$ we obtain a
homomorphism from $K_0(\OO)$ to $\on{End}_{\C}(U)[[z]]$ sending a
representation $W$ to the corresponding transfer-matrix $t_W(z)|_U$
acting on $U$. Then the above $TQ$-relations become the relations on
the joint eigenvalues of these transfer-matrices on $U$, which are the
commuting quantum Hamiltonians of the XXZ-type integrable model
associated to $U_q(\wh\g)$.

In \cite{FH}, we proved that if $U$ is a tensor product of
finite-dimensional irreducible representations of $U_q(\ghat)$, then
all eigenvalues of the transfer-matrix corresponding to the $i$th
prefundamental representation $L_{i,a}$ are polynomials in $z$ (the
degrees of these polynomials can be found from the weight of the
corresponding eigenvector) up to a universal function that is
determined by $U$. This generalizes Baxter's description of the
spectrum of the commuting quantum Hamiltonians of the XXZ model (which
corresponds to $\g=\sw_2$). For other representations, the eigenvalues
will have different analytic properties (see
e.g. \cite{BLZ1,BLZ2,BLZ3}).

\medskip

(3) Next, we discuss the {\em $QQ$-system}. To motivate it, let us
rewrite the relation \eqref{baxter} for $\g=\sw_2$ as a second-order
$q^2$-difference equation satisfied by $Q_a$ (we set
$u=1$ to simplify the formula)
\begin{equation}    \label{second}
({\mb D}^2 - T_{aq^2}{\mb D} + 1) Q_a = 0,
\end{equation}
where ${\mb D} \cdot f_a = f_{aq^2}$. There is a unique second
solution $\wt{Q}_a$ of equation \eqref{second} such that the
$q^2$-difference Wronskian of $Q_a$ and $\wt{Q}_a$ is equal to 1:
\begin{equation}    \label{wronskian}
  \wt{Q}_{aq}Q_{aq^{-1}} - \wt{Q}_{aq^{-1}}Q_{aq} = 1.
\end{equation}

It is a remarkable fact that equation \eqref{wronskian} can also be
interpreted as a relation in $K_0(\OO)$, with $\wt{Q} =
[L(\Psib_a^{-1})]$, the class of another representation of the Borel
subalgebra of $U_q(\wh{\sw}_2)$ which we denote by $L(\Psib_a^{-1})$,
up to a normalization factor. This is the $QQ$-system for $\g=\sw_2$ (it
consists of a single equation in this case).

For a general simple Lie algebra $\g$, an analogous system was
discovered by D. Masoero, A. Raimondo, and D. Valeri in
\cite{MRV1,MRV2} in the context of affine opers, which were introduced
in \cite{FF:sol}. In \cite{FH2}, motivated by \cite{MRV1,MRV2} and the
Langlands duality conjecture of \cite{FF:sol} expressing the spectra
of the quantum KdV Hamiltonians in terms of affine opers, we
interpreted the $QQ$-system of \cite{MRV1,MRV2} as a relation in
$K_0(\OO)$. It has the form
\begin{equation}    \label{QQsystem}
\wt{{\mb Q}}_{i,aq_i} {\mb Q}_{i,aq_i^{-1}} - 
a_i^{-1}  \wt{{\mb Q}}_{i,aq_i^{-1}} {\mb Q}_{i,aq_i}
\end{equation}
$$=  \prod_{j,C_{i,j} = -1} {\mb Q}_{j,a}\prod_{j,C_{i,j} = -2} {\mb Q}_{j,aq} {\mb Q}_{j,aq^{-1}}
\prod_{j,C_{i,j} = -3} {\mb Q}_{j,aq^2} {\mb Q}_{j,a}
{\mb Q}_{j,aq^{-2}},
$$
%$$\wt{\mathcal{Q}}_{i,w,aq_i} \mathcal{Q}_{i,w,aq_i^{-1}} - 
%[-w(\alpha_i)]  \wt{\mathcal{Q}}_{i,w,aq_i^{-1}}\mathcal{Q}_{i,w,aq_i}$$
%$$= [-w(\alpha_i)]_+
%   \prod_{j,C_{i,j} = -1}\mathcal{Q}_{j,w,a}\prod_{j,C_{i,j} = -2}\mathcal{Q}_{j,w,aq}\mathcal{Q}_{j,w,aq^{-1}}
%\prod_{j,C_{i,j} = -3}\mathcal{Q}_{j,w,aq^2}\mathcal{Q}_{j,w,a}\mathcal{Q}_{j,w,aq^{-2}},$$
%where $[-w(\alpha_i)]_+ = 1$ if $w(\alpha_i)\in\Delta_+$ and $[-w(\alpha_i)]_+ = - [-w(\alpha_i)]$ if $w(\alpha_i)\in \Delta_-$.
where ${\mb Q}_{i,a}$ is the class of $L(\Psib_{i,a})$ divided by its ordinary character and 
$\wt{\mb Q}_{i,a}$ is the class of another simple representation divided
by its ordinary character and by the factor $(1 - a_i^{-1})$ (where
$a_i^{-1}$ is the class of a one-dimensional
representation corresponding to the negative $i$th simple root).
(Note that in \cite{MRV1,MRV2} and
\cite{FH2}, this system was referred to as the $Q\wt{Q}$-system.)
Hence we can again
interpret it as a universal system of relations in $K_0(\OO)$, from
which the $QQ$-systems satisfied by the joint eigenvalues of the
corresponding transfer-matrices follow automatically.

In \cite{FKSZ} (see also \cite{KZ}), the $QQ$-system
  for simply-laced $\g$ was interpreted in terms of the Miura
  $(G,q)$-opers (for non-simply laced $\g$, the $QQ$-system arising
  from $(G,q)$-opers corresponds to the folded quantum integrable
  system of XXZ type introduced in \cite{FHR}).

  \medskip
  
(4) There is also a similar system (with additive shifts instead of
multiplication by $q^n$) for the XXX-type model associated to the
Yangian $Y(\g)$, which similarly
implies the $QQ$-systems on the joint eigenvalues of the corresponding
transfer-matrices. It was first constructed in \cite{MV}  in the
framework of discrete Miura opers. In the framework of representation
theory of Yangians, the $QQ$-system for $\g={\mathfrak g}{\mathfrak
  l}_n$ was proved in \cite{Bazhanov} following \cite{Tsuboi}. In the
case of $\g={\mathfrak s}{\mathfrak o}_{2n}$, analogous results were
obtained recently in \cite{ffk,EV}. These results have been
generalized (conjecturally) to all simply-laced Lie algebras in
\cite{esv}.

\begin{rem}
  In both cases (quantum affine algebras and Yangians), the
  $QQ$-system implies (under a genericity
  assumption) that the roots of the polynomial parts of the
  $Q_{i,a}$'s satisfy the so-called Bethe Ansatz equations. These
  equations were previously used to describe the eigenvalues of the
  transfer-matrices. However, we consider the $QQ$-system as more
  fundamental than Bethe Ansatz equations because the eigenvalues of
  the transfer-matrices automatically satisfy the $QQ$-system, as
  explained above, whereas Bethe Ansatz equations (for a general Lie
  algebra $\g$) originally appeared as an educated guess from
  analyzing explicit formulas for eigenvectors and eigenvalues.
  It was indeed a good guess since we can now link solutions of
  these Bethe Ansatz equations to the eigenvalues of the
  transfer-matrices via the $QQ$-system (under a genericity
  assumption).
\end{rem}

(5) In \cite{MRV1,MRV2}, the $QQ$-system \eqref{QQsystem} was obtained
on the ``dual side'' of the Langlands correspondence proposed in
\cite{FF:sol}; namely, from solutions of differential equations
defined by affine opers. In the same paper, it was shown that in fact,
there is a whole family of coupled $QQ$-systems labeled by elements of
the Weyl group $W$ of $\g$. Namely, there are natural variables
${\mb Q}_{w(\omega_i),a}$, where $w \in W, i \in I, a \in
\C^\times$, such that the old variables ${\mb Q}_{i,a}$ and $\wt{{\mb
    Q}}_{i,a}$ are identified with ${\mb Q}_{\omega_i,a}$ (i.e. $w =
e$) and ${\mb Q}_{s_i(\omega_i),a}$ (i.e. $w=s_i$, the $i$th
simple reflection from $W$), respectively. These variables satisfy the
equations
$${\mb Q}_{ws_i(\omega_i),aq_i} {\mb Q}_{w(\omega_i),aq_i^{-1}} - 
 {\mb Q}_{ws_i(\omega_i),aq_i^{-1}}{\mb Q}_{w(\omega_i),aq_i}$$
$$= 
   \prod_{j,C_{i,j} = -1}{\mb Q}_{w(\omega_j),a}\prod_{j,C_{i,j} = -2}{\mb Q}_{w(\omega_j),aq}{\mb Q}_{w(\omega_j),aq^{-1}}
\prod_{j,C_{i,j} = -3}{\mb Q}_{w(\omega_j),aq^2}{\mb Q}_{w(\omega_j),a}{\mb Q}_{w(\omega_j),aq^{-2}},$$
(here we have suppressed some coefficients analogous to
  the $a_i^{-1}$ appearing in relations \eqref{QQsystem};
  the precise form of these relations can be found in Theorem \ref{QQrel} 
  below).

These equations also appear naturally from the Miura $(G,q)$-opers
\cite{FKSZ}, and their additive analogues appear from the discrete
Miura opers \cite{MV}. The latter have been interpreted as relations
on representations of Yangians in
\cite{Tsuboi,Bazhanov,ffk,EV,esv}.

Thus, one obtains a family of $QQ$-systems labeled by $w \in W$,
with the original equations corresponding to $w=e$ (note that these
equations are coupled to each other in a non-trivial way). It is
called the {\em extended $QQ$-system}.

It is natural to ask whether the equations of the extended
$QQ$-system can be interpreted as relations in $K_0(\OO)$. In
particular, what representations from the category $\OO$ correspond to
the variables ${\mb Q}_{w(\omega_i),a}$ with $w \neq 1$? This was
our first motivating question for the present paper.

Our second motivating question was: are there $w$-analogues of
the $TQ$-systems with ${\mb Q}_{i,a}$ replaced by ${\mb
  Q}_{w(\omega_i),a}, w \in W$?

In this paper we tackle these, and some other related, questions.  We conjecture
that the answer to each of these two questions is affirmative. In
particular, we conjecture (and prove for $\g$ of rank up to 2) that if
we set ${\mb Q}_{w(\omega_i),a} = [L(\Psib_{w(\omega_i),a})]$, the class of a
particular representation from the category $\OO$ defined below, times a
constant normalization factor, then
these ${\mb Q}_{w(\omega_i),a}$'s will indeed satisfy the extended
$QQ$-system (see Conjecture \ref{extqq}).

\medskip

(6) While investigating these questions, we asked whether there is a
natural {\em action of the Weyl group} on the ring where the
$q$-characters of representations from $\OO$ take values, such that
when we apply $w \in W$ to the (renormalized) $q$-character of
$L(\Psib_{i,a}))$, we obtain the $q$-character of
$L(\Psib_{w(\omega_i),a})$. In our previous paper \cite{FH3} we showed that
one can indeed define such an action.

 More precisely, we defined an action of the Weyl group
  $W$ of $\g$ on a direct sum of completions of the ring $\mathcal{Y}
  = \mathbb{Z}[Y_{i,a}^{\pm 1}]_{i\in I, a\in\mathbb{C}^\times}$ (this
  is the ring where the $q$-characters of finite-dimensional
  representations of $U_q(\wh{\mathfrak{g}})$ take values), such that
  the subring of invariants of the action of $W$ in $\mathcal{Y}$ is
  spanned by the $q$-characters of finite-dimensional representations
  of the quantum affine algebra $U_q(\wh{\mathfrak{g}})$ (see Theorem
  \ref{mfh}\ below). Thus, the
  Grothendieck ring of the category of finite-dimensional
  representations of $U_q(\wh{\mathfrak{g}})$ is realized as the ring
  of invariants of an action of the Weyl group, much like the
  Grothendieck ring of the category of finite-dimensional
  representations of the Lie algebra $\g$.

\medskip
  
 (7) The Weyl group action that we constructed in \cite{FH3} has led
us to conjecture a {\em $q$-character formula} for the representations
$L(\Psib_{w(\omega_i),a})$ for all $w \in W$ (Conjecture
\ref{bass2}). Namely, $\chi_q(L(\Psib_{w(\omega_i),a}))$
  satisfies the $q_i^2$-difference equation
\begin{equation}    \label{TWY}
  [w(\omega_i)]\frac{\chi_q(L(\Psib_{w(\omega_i),aq_i^{-1}}))}{\chi_q(L(\Psib_{w(\omega_i),aq_i}))} = \Theta_w(Y_{i,a}),
\end{equation}
where $\Theta_w$ is the action of the element $w \in W$ that we
have constructed in \cite{FH3} and $[w(\omega_i)]$ corresponds to an
invertible one-dimensional representation defined below. 
  This fixes $\chi_q(L(\Psib_{w(\omega_i),a}))$ up to a constant
  (i.e. spectral parameter-independent) factor.

We have previously proved formula \eqref{TWY} for $w=e$ in \cite{FH}
(see formula \eqref{pref} below) and $w=s_i$ in \cite{FH2,FHR} (see
Theorem \ref{fex} below). In the present paper we prove it for all $w
\in W$ for Lie algebras $\g$ of rank 2.

In Theorem \ref{exwo} we compute $\Theta_w(Y_{i,a})$, which enables us
to pin down the $q$-character of $L(\Psib_{w(\omega_i),a})$ up to a
constant factor. The $W$-invariance of the $q$-characters of
finite-dimensional representations of $U_q(\ghat)$ proved in
\cite{FH3} then gives us (modulo Conjecture \ref{bass2}) the
sought-after extended Baxter $TQ$-relations in $K_0(\OO)$ (see
Conjecture \ref{exttq} and Corollary \ref{twoconj}).
Namely, let $V$ be an irreducible finite-dimensional representation of
$U_q(\ghat)$. Equating $[V]$ with the expression obtained by
replacing in the $q$-character $\chi_q(V)$ every $Y_{i,a}, i \in I$,
with
$$Y_{i,a} \mapsto [w(\omega_i)]
  \dfrac{[L(\Psib_{w(\omega_i),aq_i^{-1}})]}{[L(\Psib_{w(\omega_i),aq_i})]}
  $$
(and clearing the denominators), we obtain the extended Baxter
$TQ$-relation corresponding to $V$ and $w \in W$.

\medskip

(8) The third motivating question for this paper was whether the
eigenvalues of the transfer-matrices corresponding to the
representations $L(\Psib_{w(\omega_i),a}), w \in W$, acting on a tensor
product $U$ of finite-dimensional irreducible representations of
$U_q(\ghat)$ are {\em polynomials} up to an overall factor determined
by $U$. We will address this question in a follow-up paper
\cite{next}.

\medskip

(9) Our results have also led us (in Section \ref{r2s}) to formulate
conjectural $q$-character formulas for certain simple representations
of {\em shifted quantum affine algebras} introduced in \cite{FT,
  Hsh}. We prove that they are satisfied for $w=e$ and $w=s_i$ and
arbitrary $w \in W$ for Lie algebras of rank 2.  We expect analogous
formulas to hold for representations of shifted Yangians which have
recently been studied in \cite{brkl, kwwy1, nw}; in particular, in the
context of the quantization of transverse slices to Schubert varieties
in the affine Grassmannian and Coulomb branches of quiver gauge
theories.

\medskip

The paper is organized as follows. In Section \ref{back} we recall
some basic facts about quantum affine algebras $U_q(\ghat)$ and their
representations.  We also introduce the category $\OO$ of
representations of the Borel subalgebra $U_q(\wh{\mathfrak{b}})$ of
$U_q(\ghat)$ defined in \cite{HJ}.  In Section \ref{TQ}, we conjecture
the extended $TQ$-relations (Conjecture \ref{exttq}) which involve
classes of simple representations from the category $\OO$ whose
highest $\ell$-weights are obtained using a generalization of Chari's
braid group action \cite{C}. We then show how the Weyl group action
introduced in \cite{FH3} was motivated by this picture.  In Section
\ref{wgao}, we use the Weyl group action to prove Conjectures
\ref{exttq} and \ref{exttq1} when $w \in W$ is a simple
reflection. We give a conjectural $q$-character formula
  for $L(\Psib_{w(\omega_i),a})$ in Conjecture \ref{bass2} and show
  that it implies the extended Baxter $TQ$-relations (see Corollary
  \ref{twoconj}). In Section \ref{QQ}, we formulate the extended
$QQ$-system in $K_0(\OO)$ (Conjecture \ref{extqq}). In Section
\ref{shifted} we introduce the shifted quantum affine algebras and
formulate analogues of our conjectures for representations of these
algebras. In Section \ref{r2s}, we prove our conjectures for Lie
algebras of rank 2.

\medskip

\noindent {\bf Acknowledgments.} The authors were partially supported
by a grant from the France-Berkeley Fund of UC Berkeley. The second
author would like to thank the Department of Mathematics at UC
Berkeley for hospitality during his visit in the Spring of 2022 and
Bernard Leclerc for useful discussions. We also thank
  Dmytro Volin for useful correspondence.

\section{Background on quantum affine algebras}    \label{back}

In this section we collect definitions and results on quantum affine
algebras and their representations that we will need in the main body
of the paper. We refer the reader for more details to the textbook
\cite{CP} and the more recent surveys on this topic \cite{CH,L}. We
also discuss representations of the Borel subalgebra of a quantum
affine algebra (see \cite{HJ, FH} for more details). 

\subsection{Lie algebra and Weyl group}\label{Liealg}

Let $\Glie$ be a finite-dimensional simple Lie algebra of rank $n$ and
${\mathfrak{h}}$ its Cartan subalgebra. We will use the notation and
conventions of \cite{ka}.  We denote by $C = (C_{i,j})_{i,j\in I}$ its
Cartan matrix, where $I=\{1,\ldots, n\}$. We denote by $h^\vee$
(resp. $d$) the dual Coxeter number (resp. the lacing number) of
$\Glie$.

Let $\{\alpha_i\}_{i\in I}$,
$\{\alpha_i^\vee\}_{i\in I}$, $\{\omega_i\}_{i\in I}$,
$\{\omega_i^\vee\}_{i\in I}$ be the sets of simple roots, simple
coroots, fundamental weights, and fundamental coweights, respectively.
We set
$$
Q=\bigoplus_{i\in I}\Z\alpha_i, \quad Q^+=\bigoplus_{i\in
  I}\Z_{\ge0}\alpha_i, \quad P=\bigoplus_{i\in I}\Z\omega_i, \quad P^+=\bigoplus_{i\in I}\Z_{\ge0}\omega_i, \quad
P^\vee=\bigoplus_{i\in I}\Z\omega_i^\vee,
$$
and let $P_\Q := P\otimes \Q$, $P^\vee_\mathbb{Q} := P^\vee\otimes
\mathbb{Q}$.

We have the set of roots
$\Delta\subset Q$, which is decomposed as $\Delta = \Delta_+ \sqcup
\Delta_-$ where $\Delta_\pm = \pm (\Delta\cap Q^+)$.  Let
$D=\mathrm{diag}(d_1,\ldots,d_n)$ be the unique diagonal matrix such
that $B=DC$ is symmetric and $d_i$'s are relatively prime positive
integers.  

We will use the partial ordering on $P_\Q$ defined by the rule
$\omega\leq \omega'$ if and only if $\omega'-\omega\in Q^+$. The
vector space of invariant symmetric bilinear forms on $\g$ is known to
be one-dimensional. There is a unique such form $\kappa_0$
with the property that the bilinear form $(\cdot,\cdot) :=
(\kappa_0|_{\h})^{-1}(\cdot,\cdot)$ on $\h^*$ dual to the restriction
$\kappa_0|_{\h}$ of $\kappa_0$ to $\h \subset \g$
satisfies $(\alpha_i,\alpha_i)=2d_i$ ($\kappa_0$ is the form
introduced in \cite{ka} divided by $d$). We have for $i,j\in I$,
$$(\alpha_i,\alpha_j) = d_i C_{i,j} = d_j C_{j,i} \quad \text{ and } \quad
(\alpha_i,\omega_j) = d_j \delta_{i,j}.$$ Hence the bases $(\alpha_i)_{i\in I}$,
$(\omega_i/d_i)_{i\in I}$ are dual to each other with respect to this
form and $\alpha_i = \sum_{j\in I} C_{j,i}\omega_j$, 

Let $W$ be the Weyl group of $\mathfrak{g}$, generated by the simple
reflections $s_i, i\in I$. For $w\in W$, we denote by $l(w)$ its
length. The group $W$ acts on $P$, with the action determined by the
formula
$$
s_k(\omega_j) := \omega_j - \delta_{k,j}\alpha_j, \qquad j,k \in I.
$$
The form $(\cdot,\cdot)$ is invariant under this action, i.e. for any
$\lambda, \mu\in P$, we have
\begin{equation}\label{winv}(w \lambda,w \mu) = (\lambda,\mu)\text{
    for any $w\in W$.}\end{equation}
%For $w = s_k$ simple reflexion, $i,j\in I$, one has
%$$(s_k(\alpha_i), s_k(\omega_j)) 
%= (\alpha_i - C_{k,i}\alpha_k,\omega_j - \delta_{k,j}\alpha_j)
%= (\alpha_i,\omega_j) -\delta_{k,j}(d_i C_{i,j} + d_j C_{j,i} - 2r_jC_{j,i}) = %(\alpha_i,\omega_j).$$
The Weyl group invariance implies that for $i,k \in I$ and $w\in W$ we have 
$$(w\omega_i,\alpha_k) = (\omega_i,w^{-1}\alpha_k)$$
which is equal to $d_i$ multiplied by the multiplicity of $\alpha_i$ in the decomposition of the root $w^{-1}\alpha_k$. This root is positive or negative, hence the sign of $(w\omega_i,\alpha_k)$ does not depend on $i\in I$. 

Now consider a reduced decomposition $w = s_{i_1}\cdots s_{i_N}$ of a
Weyl group element $w$. Then $w^{-1} = s_{i_N}\cdots s_{i_1}$ is a
reduced decomposition of $w^{-1}$ and $s_{i_N}\cdots s_{i_L}$ is also
a reduced decomposition for all $1\leq L\leq N$.

\begin{lem}\label{sign} The multiplicity of $\omega_{i_M}$ in 
$$(s_{i_{M+1}}\cdots s_{i_N})(\omega_i)$$
is non-negative for any $i\in I$.
\end{lem}

\begin{proof}
This follows from the formula (see \cite[Lemma 3.11]{ka})
$$(s_{i_N}\cdots s_{i_M})(\alpha_{i_M})\in \Delta_-.$$
\end{proof}

%We set $$i\sim j\text{ if }C_{i,j} < 0.$$

The following lemma is also well-known. Let $W_i$ be the
  subgroup of $W$ generated by $s_j$, $j\neq i$.

\begin{lem}\label{wn} Let $w\in W$ and $i\in I$. Then

(1) $w(\alpha_i)\in \Delta_+$ if and only if $l(ws_i) = l(w) + 1$.

(2) $w(\omega_i) = \omega_i$ if and only if  $w \in W_i$,
    so the $W$-orbit of $\omega_i$ is in bijection with $W/W_i$.

(3) The $W$-orbits $\{ w(\omega_i) \}_{w \in W}$ of the fundamental
    weights are disjoint.
\end{lem}

\begin{proof}
  Part (1) follows, for example, from \cite[Lemma 1.6, Corollary
    1.7]{hu}.

  Part (2) is a particular case of \cite[Proposition 1.15]{hu}. Indeed
  $(\omega_i,\alpha_i) > 0$ and 
	$$\omega_i \in C = \{\lambda | (\lambda
  , \alpha_j) = 0\}\text{ for any $j\neq i$.}$$ 
	This proposition in the case
  of the subset $\{1, 2, ..., n\} \setminus \{i\} \subset I$ implies
  the result.

To prove part (3), note that $w(\omega_i) = \omega_j$
implies that $\omega_j\preceq \omega_i$ as $\omega_i$ is the highest
weight of a fundamental representation of $\mathfrak{g}$.
\end{proof}

Denote the $W$-orbit of $\omega_i$ in $P$ by ${\mc
    P}_{\omega_i}$. This is the set of extremal weights of the $i$th
  fundamental representation ${\mc V}_{\omega_i}$ of $\g$. Each of
  them has multiplicity one. If these are all the weights of ${\mc
    V}_{\omega_i}$, then ${\mc V}_{\omega_i}$ is called minuscule
  (for example, this is so for all $i$ if $\g$ is of type $A$).
  
Denote by $w_0$ the longest element of the Weyl group. We have the bar
involution of the set $I$ defined by $w_0(\alpha_i) =
-\alpha_{\overline{i}}, i \in I$. Note that this implies $w_0(\omega_i) =
  -\omega_{\overline{i}}, i\in I$. This is the lowest weight of the
  fundamental representation ${\mc V}_{\omega_i}$.
  
\subsection{Quantum affine algebras and their Borel subalgebras}

Let $\wh{\Glie}$ be the untwisted affine Kac-Moody Lie algebra
associated to $\Glie$ and $(C_{i,j})_{0\leq i,j\leq n}$ the
corresponding indecomposable Cartan matrix. Let $a_0,\cdots,a_n$ be
the Kac labels defined in \cite{ka}, pp. 55-56. We have $a_0 = 1$ and
we set $\alpha_0 = -(a_1\alpha_1 + a_2\alpha_2 + \cdots +
a_n\alpha_n)$.

Throughout this paper, we fix a non-zero complex number $q$ which is
not a root of unity. {\em We also choose $h\in\CC$
such that $q = e^h$ and define $q^r \in \C^\times$ for any
$r\in\Q$ as $e^{hr}$}. Since $q$ is not a root of unity, the map $\Q \to
  \C^\times$ sending $r \mapsto q^r$ is injective.

We will use the standard symbols for $q$-integers
\begin{align*}
[m]_z=\frac{z^m-z^{-m}}{z-z^{-1}}, \quad
[m]_z!=\prod_{j=1}^m[j]_z,
 \quad 
\qbin{s}{r}_z
=\frac{[s]_z!}{[r]_z![s-r]_z!}. 
\end{align*}
We set $q_i=q^{d_i}$.

The quantum affine algebra of level 0 (also known as the quantum loop
algebra) $U_q(\wh{\Glie})$ is the $\C$-algebra
with the generators $e_i,\ f_i,\ k_i^{\pm1}$ ($0\le i\le n$) and the
following relations for $0\le i,j\le n$:
\begin{align*}
&k_ik_j=k_jk_i,\quad k_0^{a_0}k_1^{a_1}\cdots k_n^{a_n}=1,\quad
&k_ie_jk_i^{-1}=q_i^{C_{i,j}}e_j,\quad k_if_jk_i^{-1}=q_i^{-C_{i,j}}f_j,
\\
&[e_i,f_j]
=\delta_{i,j}\frac{k_i-k_i^{-1}}{q_i-q_i^{-1}},
\\
&\sum_{r=0}^{1-C_{i.j}}(-1)^re_i^{(1-C_{i,j}-r)}e_j e_i^{(r)}=0\quad (i\neq j),
&\sum_{r=0}^{1-C_{i.j}}(-1)^rf_i^{(1-C_{i,j}-r)}f_j f_i^{(r)}=0\quad (i\neq j)\,.
\end{align*}
Here $x_i^{(r)}=x_i^r/[r]_{q_i}!$ ($x_i=e_i,f_i$). 
The algebra $U_q(\wh{\Glie})$ has a Hopf algebra structure such that
\begin{align*}
&\Delta(e_i)=e_i\otimes 1+k_i\otimes e_i,\quad
\Delta(f_i)=f_i\otimes k_i^{-1}+1\otimes f_i,
\quad
\Delta(k_i)=k_i\otimes k_i\,\text{ for $0\leq i\leq n$.}
\end{align*}
The algebra $U_q(\wh{\Glie})$ can also be presented in terms of the Drinfeld
generators (see \cite{Dri2,bec} for details)
\begin{align}\label{dgen}
  x_{i,r}^{\pm}\ (i\in I, r\in\Z), \quad \phi_{i,\pm m}^\pm\ (i\in I,
  m\in\mathbb{Z}), \quad k_i^{\pm 1}\ (i\in I).
\end{align}
We will use the generating series $(i\in I)$
\begin{equation}\label{phrel}
  \phi_i^\pm(z) = \sum_{m\in \mathbb{Z}}\phi_{i, m}^\pm z^{ m} =
k_i^{\pm 1}\on{exp}\left(\pm (q_i - q_i^{-1})\sum_{r > 0} h_{i,\pm
  r} z^{\pm r} \right).\end{equation} In particular, $\phi_{i,\pm
  m}^\pm = 0$ for $m < 0$, $i\in I$, and $\phi_{i,0}^\pm
  = k_i^{\pm 1}$, so $\phi_{i,0}^+\phi_{i,0}^- = 1$.

The algebra $U_q(\wh{\Glie})$ has a $\ZZ$-grading defined by
$\on{deg}(e_i) = \on{deg}(f_i) = \on{deg}(k_i^{\pm 1}) = 0$ for $i\in
I$ and $\on{deg}(e_0) = - \on{deg}(f_0) = 1$.  It satisfies
$\on{deg}(x_{i,m}^\pm) = \on{deg}(\phi_{i,m}^\pm) = m$ for $i\in I$,
$m\in\ZZ$. For $a\in\CC^\times$, there is an automorphism
$$\tau_a : U_q(\wh{\Glie})\rightarrow U_q(\wh{\Glie})$$ 
such that $\tau_a(g) = a^m g$ for any element $g$ of degree
$m\in\ZZ$.

\begin{defi} The {\em Borel subalgebra} $U_q(\wh\bb)$ is 
the subalgebra of $U_q(\wh{\Glie})$ generated by $e_i$ 
and $k_i^{\pm1}$ with $0\le i\le n$. 
\end{defi}

This is a Hopf subalgebra of $U_q(\wh{\Glie})$. 
The algebra $U_q(\wh\bb)$ contains the Drinfeld generators
$x_{i,m}^+$, $x_{i,r}^-$, $k_i^{\pm 1}$, $\phi_{i,r}^+$ 
with $i\in I$, $m \geq 0$ and $r > 0$. 

\begin{rem}
For $\g=\sw_2$, the algebra $U_q(\wh\bb)$ is generated by the Drinfeld
generators, but for other Lie algebras $\g$ this is not the case (see
\cite[Section 2.3]{HJ} and \cite{bec}).
\end{rem}
  
%\begin{align}
%\omega\leq \omega' \quad \text{if $\omega \omega'^{-1}$ 
%is a product of $\{\ga_i^{-1}\}_{i\in I}$}.
%\label{partial}
%\end{align}

\subsection{Category $\mathcal{O}$ for representations of Borel
  subalgebras}    \label{catO}

%Let $\tb
%    \subset U_q(\wh\bb)$ be the subalgebra generated by $\{k_i^{\pm
%      1}\}_{i\in I}$.
%
%    \medskip
%
%    {\color{red} {\bf Do we ever use $\tb$? I think we only use
%        $\tb^\times$. If so, I suggest removing the previous
%        sentence.}}
%
%    \medskip

\begin{defi} We define $\tb^\times := \on{Maps}(I,q^\Q) =
      \bigl(q^\Q\bigr)^I$, where $q^\Q := \{ e^{rh}, r \in \Q \}$, as a
      group with pointwise multiplication.

We define
a group homomorphism
$\overline{\phantom{u}}:P_\Q \longrightarrow
\tb^\times$ by setting
$$\overline{r \omega_i}(j):=q_i^{r \delta_{i,j}} = e^{r d_i h
  \delta_{i,j}}, \qquad i,j \in I, r \in \Q.
$$
\end{defi}

Since we have assumed that $q$ is not a root of unity, this is
actually an isomorphism, with the inverse map $\ol\varpi: \tb^\times
\to P_\Q$ defined by the formula
$$
\psi_i = e^{d_ih(\ol\varpi({\mbox{\boldmath$\psi$}}))(\alpha_i^\vee)}, \qquad
    \forall {\mbox{\boldmath$\psi$}} = (\psi_i)_{i \in I} \in
    \tb^\times.
$$

The partial ordering on $P$ induces a partial ordering on
$\tb^\times$.

\medskip

For a $U_q(\wh\bb)$-module $V$ and $\omega\in \tb^\times$, we set
\begin{align}
V_{\omega}:=\{v\in V \mid  k_i\, v = \omega(i) v\ (\forall i\in I)\}\,,
\label{wtsp}
\end{align}
and call it the weight subspace of weight $\omega$. 
For any $i\in I$, $r\in\ZZ$ we have $\phi_{i,r}^\pm (V_\omega)\subset V_\omega$
and $x_{i,r}^\pm (V_{\omega}) \subset V_{\omega \ga_i^{\pm 1}}$.
We say that $V$ is Cartan-diagonalizable 
if $V=\underset{\omega\in \tb^\times}{\bigoplus}V_{\omega}$.

\begin{defi} A series $\Psib=(\Psi_{i, m})_{i\in I, m\geq 0}$ 
of complex numbers such that $\Psi_{i,0}
  \in q^\Q$ for all $i\in I$ is called an $\ell$-weight.

We denote by $\tb^\times_\ell$ the set of $\ell$-weights. We have a
natural projection $\jmath: \tb^\times_\ell \to \tb^\times$ sending
$(\Psi_{i, m})_{i\in I, m\geq 0}$ to $(\Psi_{i, 0})_{i\in I}$ and a
natural inclusion $\imath: \tb^\times \to \tb^\times_\ell$.
\end{defi}

We will identify the collection $(\Psi_{i, m})_{m\geq 0}$ with its
generating series and use this to represent every $\ell$-weight as an
$I$-tuple of formal power series in $z$:
\begin{align*}
\Psib = (\Psi_i(z))_{i\in I},
\qquad
\Psi_i(z) := \underset{m\geq 0}{\sum} \Psi_{i,m} z^m.
\end{align*}
Since each $\Psi_i(z)$ is an invertible formal power series,
$\tb^\times_\ell$ has a natural group structure.  We also have a
surjective homomorphism of groups $\varpi :
  \tb^\times_\ell\rightarrow P_\Q$ defined as the composition of the
  projection $\jmath: \tb^\times_\ell \to \tb^\times$ and the map
  $\ol\varpi: \tb^\times \to P_\Q$ defined above.
  
\begin{defi} A $U_q(\wh\bb)$-module $V$ is said to be 
of highest $\ell$-weight 
$\Psib\in \tb^\times_\ell$ if there is $v\in V$ such that 
$V =U_q(\wh\bb)v$ 
and the following equations hold:
\begin{align*}
e_i\, v=0\quad (i\in I)\,,
\qquad 
\phi_{i,m}^+v=\Psi_{i, m}v\quad (i\in I,\ m\ge 0)\,.
\end{align*}
\end{defi}

The $\ell$-weight $\Psib\in \tb^\times_\ell$ is uniquely determined by $V$. 
It is called the highest $\ell$-weight of $V$. 
The vector $v$ is said to be a highest $\ell$-weight vector of $V$.

%{\color{red} The fact that $U_q(\wh\bb)$ is generated by the
%  Drinfeld generators (see above) immediately implies the following
%  result.}

\begin{lem}\cite[Section 3.2]{HJ}\label{simple} 
For any $\Psib\in \tb^\times_\ell$, there exists a simple highest
$\ell$-weight module $L(\Psib)$ over $U_q(\wh\bb)$ of highest
$\ell$-weight $\Psib$. This module is unique up to isomorphism.
\end{lem}

%The submodule of $L(\Psib)\otimes L(\Psib')$ generated by the tensor
%product of the highest $\ell$-weight vectors is of highest
%$\ell$-weight $\Psib\Psib'$. In particular, $L(\Psib\Psib')$ is a
%subquotient of $L(\Psib)\otimes L(\Psib')$.

\begin{defi}\cite{HJ}
For $i\in I$ and $a\in\CC^\times$, let 
\begin{align}
L_{i,a}^\pm = L(\Psib_{i,a})
\quad \text{where}\quad 
(\Psib_{i,a})_j(z) = \begin{cases}
(1 - za)^{\pm 1} & (j=i)\,,\\
1 & (j\neq i)\,.\\
\end{cases} 
\label{fund-rep}
\end{align}
\end{defi}
We call $L_{i,a}^+$ (resp. $L_{i,a}^-$) a {\em positive (resp. negative)
prefundamental representation}.

%For $i\in I$ and $a\in\CC^\times$, we will also consider the representation introduced in \cite{FH2} : 
%$$X_{i,a} = L(\wt{\Psib}_{i,a})\text{ where }$$
%$$\wt{\Psib}_{i,a} = \Psib_{i,a}^{-1}
% \left(\prod_{j|C_{i,j} = -
%     1}\Psib_{j,aq_i}\right)\left(\prod_{j|C_{i,j} =
%     -2}\Psib_{j,a}\Psib_{j,aq^2}\right)\left(\prod_{j|C_{i,j} =
%     -3}\Psib_{j,aq^{-1}}\Psib_{j,aq}\Psib_{j,aq^3}\right).$$
 For $\lambda \in P_\Q$, we define the $1$-dimensional
  representation
$$[\lambda] = L(\Psib_\lambda)
\quad \text{with}\quad 
\Psib_\lambda := \imath(\ol\lambda),$$
where $\imath: \tb^\times \to \tb^\times_\ell$ is the inclusion.

For $a\in \C^\times$, the subalgebra $U_q(\wh\bb)$ is stable under
$\tau_a$. Denote its restriction to $U_q(\wh\bb)$ by the same letter.
Then the pullbacks of the $U_q(\wh\bb)$-modules $L_{i,b}^\pm$ by
$\tau_a$ is $L_{i,ab}^\pm$.

For $\lambda\in \tb^\times$, let $D(\lambda )=
\{\omega\in \tb^\times \mid \omega\leq\lambda\}$. Consider the category $\mathcal{O}$ introduced in \cite{HJ}.

\begin{defi}\label{defo} A $U_q(\wh\bb)$-module $V$ 
is said to be in category $\mathcal{O}$ if

i) $V$ is Cartan-diagonalizable;

ii) for all $\omega\in \tb^\times$ we have 
$\dim (V_{\omega})<\infty$;

iii) there exist a finite number of elements 
$\lambda_1,\cdots,\lambda_s\in \tb^\times$ 
such that the weights of $V$ are in 
$\underset{j=1,\cdots, s}{\bigcup}D(\lambda_j)$.
\end{defi}

The next result follows directly from the definition.

\begin{lem}
  The category $\mathcal{O}$ is a monoidal category.
\end{lem}

\begin{defi}
We define $\mfr$ as the subgroup of $\tb^\times_\ell$ consisting of
$\Psib$ such that $\Psi_i(z)$ is an expansion in positive powers of
$z$ of a rational function in $z$ for any $i\in I$.
\end{defi}

\begin{thm}\label{class}\cite{HJ} Let $\Psib\in\tb^\times_\ell$. 
The simple module $L(\Psib)$ is in $\mathcal{O}$ if and only if
$\Psib\in \mfr$.
\end{thm}

\begin{defi}    \label{defE}
We define $\mathcal{E}$ 
as the additive group of maps $c : P_\Q \rightarrow \ZZ$ 
whose support 
$$\on{supp}(c) = \{\omega\in P_\Q,c(\omega) \neq 0\}$$ 
is contained in 
a finite union of sets of the form $D(\mu)$. 
For $\omega\in P_\Q$, we define $[\omega] = \delta_{\omega,.}\in\mathcal{E}$.

For $V$ in the category $\mathcal{O}$ we define the character of
$V$ to be the following element of $\mathcal{E}$
\begin{align}
\chi(V) = \sum_{\omega\in\tb^\times} 
\on{dim}(V_\omega) [\omega]\,.
\label{ch}
\end{align}
\end{defi}

In the same way as in the case of the ordinary category $\mathcal{O}$
of $\g$-modules (see \cite[Section 9.6]{ka}), one shows that the
multiplicity of any simple module in a given module from the category
$\mathcal{O}$ is well-defined. One uses this fact in the following
definition \cite[Section 3.2]{HL}.

\begin{defi}    \label{K0}
We define $K_0(\mathcal{O})$ as the completion of the Grothendieck
ring of the category $\OO$ whose elements are sums (possibly infinite)
$$\sum_{\Psibs \in \mfr} \lambda_{\Psibs} [L(\Psib)],$$ 
such that the $\lambda_{\Psibs}\in\ZZ$ satisfy
$$
\sum_{\Psibs \in
  \mfr, \omega\in P_\Q} |\lambda_{\Psibs}|
\on{dim}((L(\Psib))_\omega) [\omega] \in \mathcal{E}.
$$
\end{defi}

We naturally identify $\mathcal{E}$ with 
the Grothendieck ring of the subcategory of the category
$\mathcal{O}$ whose obtains are representations with constant
$\ell$-weights. Its
simple objects are the $[\omega]$, $\omega\in P_\Q$. Thus
as in \cite[Section 9.7]{ka} we will regard elements of $\mathcal{E}$
as formal sums
\[
 c = \sum_{\omega\in \text{Supp}(c)} c(\omega)[\omega].
\]
The multiplication is given by $[\omega][\omega'] = [\omega+\omega']$ and $\mathcal{E}$
is regarded as a subring of $K_0(\mathcal{O})$. 

The character defines a ring homomorphism $\chi:
K_0(\mathcal{O})\rightarrow \mathcal{E}$ which is not injective.

\subsection{Monomials and finite-dimensional representations}\label{fdrep}

Following \cite{Fre}, we give the following definition.

\begin{defi}    \label{defY}
Define the ring of Laurent polynomials $\Yim =
\ZZ[Y_{i,a}^{\pm 1}]_{i\in I,a\in\CC^\times}$ in the variables
$\{Y_{i,a}\}_{i\in I, a\in \C^\times}$. Define $\mathcal{M}$ as the
multiplicative group of monomials in $\Yim$.

For $i\in I, a\in\CC^\times$, define
$A_{i,a}\in\mathcal{M}$ by the formula
\begin{multline}    \label{Ai}
A_{i,a} := \\ Y_{i,aq_i^{-1}}Y_{i,aq_i}
\Bigl(\prod_{\{j\in I|C_{j,i} = -1\}}Y_{j,a}
\prod_{\{j\in I|C_{j,i} = -2\}}Y_{j,aq^{-1}}Y_{j,aq}
\prod_{\{j\in I|C_{j,i} =
-3\}}Y_{j,aq^{-2}}Y_{j,a}Y_{j,aq^2}\Bigr)^{-1}\,.
\end{multline}

%For a monomial 
%$m = \prod_{i\in I, a\in\CC^\times}Y_{i,a}^{u_{i,a}}$, 
%we consider its `evaluation on $\phi^+(z)$'. 
%By definition it is an element 
%$m(\phi(z))\in\mfr$  
%given by
%$$m\bigl(\phi(z))=
%\prod_{i\in I, a\in\CC^\times}
%\left(Y_{i,a}(\phi(z))\right)^{u_{i,a}}\text{ where }
%\Bigl(Y_{i,a}\bigl(\phi(z)\bigr)\Bigr)_j
%=\begin{cases}
%\displaystyle{q_i\frac{1-a q_i^{-1}z}{1-aq_iz}}& (j=i),\\
%1 & (j\neq i).\\
%\end{cases}$$

Define an injective group homomorphism 
$\mathcal{M}\rightarrow \mfr$ by 
$$Y_{i,a}\mapsto [\omega_i]\Psib_{i,aq_i^{-1}}\Psib_{i,aq_i}^{-1}.$$
\end{defi}

Note that $\varpi(Y_{i,a}) = \omega_i$ and $\varpi(A_{i,a}) =
\alpha_i$.

We will identify a monomial $m\in\mathcal{M}$ with its image in
$\mfr$. In particular, for any $m\in\mathcal{M}$, we will
  denote by $L(m)$ the corresponding simple $U_q(\wh\bb)$-module.

%\begin{rem}
%Note that we could introduce the $\Psib_{i,a}$'s as
%formal variables satisfying the relation $Y_{i,a} =
%[\omega_i]\Psib_{i,aq_i^{-1}}\Psib_{i,aq_i}^{-1}$ and consider the
%extension ring $\Yim' = \ZZ[\Psib_{i,a}^{\pm 1}]_{i\in
%I,a\in\CC^\times}$ of $\Yim$.
%\end{rem}

Let $\mathcal{C}$ be the category of type $1$ finite-dimen\-sional
representations of $U_q(\wh{\Glie})$. For brevity, we will omit ``type
1'' and refer to objects of the category ${\mc C}$ simply as
finite-dimensional representations of $U_q(\wh{\Glie})$. For $j\in I$,
a monomial in $m\in\mathcal{M}$ is said to be $j$-{\em dominant}
(resp. $j$-antidominant) if the powers of the variables $Y_{j,b}$,
$b\in\mathbb{C}^\times$, occurring in this monomial are all positive
(resp. negative). A monomial $m\in\mathcal{M}$ is said to be {\em
  dominant} if it is $j$-dominant for any $j\in I$. We
  will denote by ${\mc M}^+$ the set of dominant monomials. This is a
  subsemi-group of ${\mc M}$.

Part (1) of the following theorem was established in \cite{CP}
following \cite{Dri2}. Parts (2) and (3) were proved in \cite{HJ}
using the result of part (1).
  
\begin{thm}
(1) Every simple module in the category $\mathcal{C}$ is of the form
  $L(m), m \in {\mc M}^+$.

(2) For $m \in \mathcal{M}$, the simple $U_q(\wh\bb)$-module $L(m)$ is
  finite-dimensional if and only if $m \in {\mc M}^+$. Moreover, the
  action of $U_q(\wh\bb)$ on $L(m)$ can be uniquely extended to an
  action of $U_q(\wh{\Glie})$ in this case.

(3) Every finite-dimensional simple module in the category $\OO$ is of
  the form $L(m) \otimes [\omega], m \in {\mc M}^+, \omega \in P_\Q$.
\end{thm}

%According to \cite[Remark 3.11]{FH}, up to tensoring with
%one-dimensional representations, the modules described in this theorem
%are all simple finite-dimensional modules in the category
%$\mathcal{O}$.

%Note that if $\Psib$ is a monomial in the variables
%$$\wt{Y}_{i,a} = [-\omega_i]Y_{i,a},$$ 
%then $L(\Psib)$ is finite-dimensional. We will also use in the
%following notation:
%$$\wt{A}_{i,a} =
%\Psib_{i,aq_i^{-2}}\Psib_{i,aq_i^2}^{-1}\left(\prod_{j\sim i, r_j >
%    1}\Psib_{j,aq_j^{-1}}^{-1}\Psib_{j,aq_j}\right)
%\left(\prod_{j\sim i, r_j =
%    1}\Psib_{j,aq_i^{-1}}^{-1}\Psib_{j,aq_i}\right) =
%[-\alpha_i]A_{i,a}.$$

In particular, the finite-dimensional fundamental representation
$L(Y_{i,a}), i\in I$, $a\in\CC^\times$, is called the $i$th
fundamental representation. The classes of these representations
freely generate the Grothendieck ring $K_0({\mc C})$.

%The simple tensor product of a KR-module
%by a one-dimensional representation $[\omega]$, $\omega\in P$, will
%also be called a KR-module.

\subsection{$q$-characters}    \label{qchar}

\begin{defi}
For a $U_q(\wh\bb)$-module $V$ and
$\Psib\in\tb_\ell^\times$, the subspace of $V$,
\begin{align}
V_{\Psibs} :=
\{v\in V\mid
\exists p\geq 0, \forall i\in I, 
\forall m\geq 0,  
(\phi_{i,m}^+ - \Psi_{i,m})^pv = 0\}
\label{l-wtsp} 
\end{align}
is called the $\ell$-weight subspace of $V$ of $\ell$-weight $\Psib$.
\end{defi}

\begin{thm}\cite{HJ} For $V$ in category $\mathcal{O}$, $V_{\Psibs}\neq
  0$ implies $\Psib\in\mfr$.
\end{thm}

%A natural ring of formal power series $\mathcal{E}_{\ell}\subset \mathbb{Z}^{\mathfrak{r}}$ is defined in \cite{HJ} so that $\varpi$ is naturally extended to a surjective homomorphism 
%$\varpi : \mathcal{E}_\ell\rightarrow \mathcal{E}$.

We define $\mathcal{E}_\ell$ as the additive group of maps $c : \mathfrak{r} \rightarrow \ZZ$ 
so that the image of its support 
$$\on{supp}(c) = \{\Psib\in \mathfrak{r},c(\Psib) \neq 0\}$$ 
by $\varpi$ is contained in 
a finite union of sets of the form $D(\mu)$. The group
$\mathcal{E}_\ell$ is a completion of the group ring of $\mfr$ and the
  ring structure on the latter extends to ${\mc E}_\ell$ by
  continuity. Moreover, the homomorphism
$\varpi$ naturally extends to a surjective homomorphism 
\begin{equation}    \label{varpi}
  \varpi : \mathcal{E}_\ell \twoheadrightarrow \mathcal{E},
\end{equation}
where $\mathcal{E}$ is defined in Definition \ref{defE}.
We also have a natural inclusion $\mathcal{E} \hookrightarrow
\mathcal{E}_\ell$.

Since $\Yim$ is the group ring of ${\mc M}$ (see Definition
\ref{defY}), the injective group homomorphism ${\mc M} \to \mfr$
extends to an injective ring homomorphism $\Yim \to {\mc E}_\ell$.

%Let 
%$\mathcal{E}_\ell\subset \Z^{\mfr}$
%be the ring of maps
%$c : \mfr\rightarrow \ZZ$  
%satisfying $c(\Psib) = 0$ for all 
%$\Psib$ such that $\varpi(\Psib)$ is outside a finite union
%of sets of the form $D(\mu)$ and such that for 
%each $\omega\in \tb^*$, there are finitely many $\Psib$
%such that $\varpi(\Psib) = \omega$ and $c(\Psib)\neq 0$.

For $\Psib\in\mfr$, we define $[\Psib] =
\delta_{\Psibs,.}\in\mathcal{E}_\ell$.

Let $V$ be a $U_q(\wh\bb)$-module in category $\mathcal{O}$. 
Following \cite{Fre, HJ}, we define the {\em $q$-character} of $V$ as
\begin{align}
\chi_q(V) = 
\sum_{\Psibs\in\mfr}  
\mathrm{dim}(V_{\Psibs}) [\Psib]\in \mathcal{E}_\ell\,.
\label{qch}
\end{align}
		
The following is proved in \cite{Fre} and \cite[Section 3.4]{HJ}.

\begin{prop}    \label{inject}
Formula \ref{qch} defines an injective ring homomorphism
$$\chi_q : K_0(\mathcal{O})\rightarrow \mathcal{E}_\ell.$$
\end{prop}

Note that we have $\chi(V) = \varpi(\chi_q(V))$ for any representation
$V$ from the category $\mathcal{O}$.

Slightly abusing notation, we will write $[\omega]$ for
  both the one-dimensional representation in the category ${\mc O}$
  corresponding to $\omega \in P_{\Q}$ and its $q$-character.

%The $q$-character map separates isomorphism classes of simple modules
%and the $q$-characters of simple modules are linearly independent.

 Recall that $\mathcal{M}$ is the multiplicative group of
  monomials of $\Yim$ and the injective group homomorphism ${\mc M}
  \to \mfr$ from Definition \ref{defY}, as well as the injective ring
  homomorphism $\Yim \to {\mc E}_\ell$. We will identify ${\mc M}$ and
  $\Yim$ with their images in $\mfr$ and ${\mc E}_\ell$, respectively,
  under these homomorphisms.

\begin{thm}\cite{Fre, Fre2}

(1) For any finite-dimensional representation $V$ of
$U_q(\wh{\Glie})$, all $\ell$-weights appearing in $V$
are in the image of ${\mc M}$ in $\mfr$.

(2) The $q$-character $\chi_q(V)$ of $V$ is an element of $\Yim =
\mathbb{Z}[Y_{i,a}^{\pm 1}]_{i\in I, a\in\mathbb{C}^\times} \subset
       {\mathcal E}_\ell$.

(3) For any simple finite-dimensional representation $L(m)$ of
       $U_q(\wh{\Glie})$, we have
$$\chi_q(L(m))\in m (1 + \ZZ[A_{i,a}^{-1}]_{i\in I, a\in\CC^\times}).$$
\end{thm}

Let $\varphi:
  \mathcal{M}\rightarrow P$ be the restriction of the homomorphism
  $\varpi$ to ${\mc M}$. Explicitly, $\varphi$ sends a monomial $m =
  \prod_{i\in I,
    a\in\CC^\times}Y_{i,a}^{u_{i,a}(m)}$ in $\mathcal{M}$ to the
  corresponding $\g$-weight
$$\varphi(m) := \sum_{i\in I, a\in\CC^\times} u_{i,a}(m) \omega_i\in
  P.$$
  For example, for $i\in I, a\in\CC^\times$, we have
  $\varphi(Y_{i,a}) = \omega_i$ and $\varphi(A_{i,a}) = \alpha_i$, where
  $A_{i,a}\in\mathcal{M}$ is defined by formula \eqref{Ai}. In
  particular, if $m \in {\mc M}^+$, then $\varphi(m) \in P^+$.

The next result follows directly from the definitions.

\begin{lem}    \label{Winv}
(1) For $m \in {\mc M}^+$, let $L(m)$ be the corresponding simple
$U_q(\wh{\mathfrak{g}})$-module, and set $\lambda := \varphi(m) \in P^+$.
The restriction of $L(m)$ to $U_q(\mathfrak{g})$ is isomorphic to a
direct sum of simple $U_q(\mathfrak{g})$-modules of the form
\begin{equation}
{\mc V}^q_\la \bigoplus \left( \bigoplus_{\mu < \la} ({\mc V}^q_\mu)^{\oplus
  c_{\la\mu}} \right),
\end{equation}
where ${\mc V}^q_\mu, \mu \in P^+$, denotes the simple
$U_q(\wh{\mathfrak{g}})$-module with highest weight $\mu$, and we use
the standard partial ordering $<$ on $P^+$ (see Section \ref{Liealg}).

(2) For any
finite-dimensional $U_q(\wh{\mathfrak{g}})$-module $V$, the ordinary
character of its restriction to $U_q(\mathfrak{g})$ is equal to
$\varphi(\chi_q(V))$.
\end{lem}

We will also need the following result.

\begin{thm}\cite{FH}\label{formuachar} For any $a\in\CC^\times$, $i\in I$ we have
$$\chi_q(L_{i,a}^+) = \left[\Psib_{i,a}\right] \chi(L_{i,a}^+)$$
       and
         $$
         \chi(L_{i,a}^+) = \chi(L_{i,a}^-).
         $$
\end{thm}

\subsection{Chari's braid group action and extremal
    monomials}\label{secextm}

 Consider the simple finite-dimensional
 $U_q(\wh{\mathfrak{g}})$-module $L(m), m \in {\mc M}^+$, and set
 $\la=\varphi(m) \in P^+$. Lemma \ref{Winv} implies that the ordinary
 character of the restriction of $L(m)$ to $U_q(\g)$ is invariant
 under the action of the Weyl group $W$. Furthermore, Lemma \ref{Winv}
 implies that for any $w\in W$, the dimension of the weight subspace
 of $L(m)$ of weight $w(\omega)$ is equal to $1$, and it is spanned by an
 extremal weight vector (in the sense of \cite{kas}) which belongs to
 the $U_q(\mathfrak{g})$-submodule ${\mc V}^q_\la$ generated by the
 highest weight vector of $L(m)$. We will denote this vector by
 $v_w$. (In particular, if $w=e$, the identity element of the Weyl group
 $W$, the vector $v_e$ is the highest weight vector of $L(m)$.) Hence
 this weight subspace is also an $\ell$-weight subspace, and $v_w$ is
 an $\ell$-weight vector.

 V. Chari introduced in \cite{C} a braid group action on
  $\Yim = \mathbb{Z}[Y_{i,a}^{\pm 1}]_{i\in I, a\in\mathbb{C}^\times}$
  and used it to compute the monomial in $\Yim$ expressing the
  $\ell$-weight of $v_w$.

Namely, Chari defined the operators $T_i, i \in I$, on $\Yim$
acting by the formula
\begin{equation}    \label{TiYj}
T_i(Y_{i,a}) = Y_{i,a}A_{i,aq_i}^{-1}\text{ and }T_i(Y_{j,a}) =
Y_{j,a} \text{ if $j\neq i$}
\end{equation}
for all $i,j\in I $ and
$a\in\mathbb{C}^\times$. The operator $T_i$ is not an involution, but
it was established in \cite{C} (a closely related result
  was proved earlier in \cite{BP}) that the $T_i$'s satisfy the
relations of the braid group corresponding to $\g$. 
  Hence, for any $w\in W$, there is a well-defined operator $T_w$
  acting on $\Yim$ by the formula
  $$
  T_w = T_{i_1} T_{i_2} \ldots T_{i_k}
  $$
  for any reduced decomposition $w = s_{i_1} \ldots s_{i_k}$.

\begin{thm}\cite{C, CM}\label{occurlm} The $\ell$-weight of $v_w \in
  L(m)$ is $T_w(m)$. In particular, $T_w(m)$ occurs in $\chi_q(L(m))$
  with multiplicity $1$.
\end{thm}

Recall the set ${\mc P}_{\omega_i} = W \cdot \omega_i \subset P$. By
Lemma \ref{wn},(ii), ${\mc P}_{\omega_i} \simeq W/W_i$, where $W_i$ is
the subgroup of $W$ generated by $s_j, j \neq i$.

  \begin{lem}    \label{welldef}
For fixed $i \in I$ and $a \in \C^\times$, the monomial $T_w(Y_{i,a}),
w \in W$, only depends on the class of $w$
in $W/W_i$. Hence there is a one-to-one correspondence
$$
\{ T_w(Y_{i,a}) \}_{w\in W} \simeq {\mc P}_{\omega_i} \qquad
T_w(Y_{i,a}) \leftrightarrow w(\omega_i).
$$
  \end{lem}

  \begin{proof} It follows from Proposition 3.1 in \cite{C} 
		that $T_w(Y_{i,a})$ depends only on $w(\omega_i) \in
                P$.
  \end{proof}
    
Lemmas \ref{welldef} and \ref{wn},(3) imply that the following
notation is well-defined:
\begin{equation}    \label{Ywia}
  Y_{w(\omega_i),a} := T_w(Y_{i,a}).
\end{equation}

\begin{example} (i) For $w = e$, $Y_{\omega_i,a} = Y_{i,a}$.

(ii) Let $w = s_i$ for some $i\in I$. For $k\in I$ and
  $a\in\mathbb{C}^\times$, we have
$$T_i(Y_k) = Y_{s_i(\omega_k),a} = Y_{k,a} A_{i,aq_i}^{-\delta_{i,k}}.$$

(iii)  By \cite{Fre2},
we have $Y_{w_0(\omega_i),a} = Y_{\overline{i},aq^{d h^\vee}}^{-1}$.
\end{example}

It is well-known that the subgroup $\bigoplus_{i\in I} \mathbb{Z} (d +
1 - d_i) \omega_i \subset P$ is stable under the action of the
Weyl group $W$. Similarly, we obtain the following result. For $i\in I$ and
$a\in\mathbb{C}^\times$  let us set
\begin{align}\label{wia} W_{i,a} = \begin{cases} Y_{i,a}&\text{ if $d_i = d$,}
\\ Y_{i,aq^{-1}}Y_{i,aq}&\text{ if $ d_i = d - 1$,} \\Y_{i,aq^{-2}}Y_{i,a}Y_{i,aq^2}&\text{ if $d_i = d - 2$.} \end{cases}\end{align}

\begin{lem} The subring
$$\mathbb{Z}[W_{i,a}^{\pm 1}]_{i\in I, a\in\mathbb{C}^\times} \subset
  \mathbb{Z}[Y_{i,a}^{\pm 1}]_{i\in I, a\in\mathbb{C}^\times} = \Yim$$
is stable under the action of the operators $T_w$, $w\in W$.
\end{lem}

\begin{proof} It suffices to check this for the operators $T_i, i \in
  I$. We have
$$T_i(W_{j,a}) = W_{j,a} U_{i,aq_i}^{-\delta_{i,j}}$$ where
\begin{align} U_{i,a} = \begin{cases} A_{i,a}&\text{ if $d_i = d$,}
\\ A_{i,aq^{-1}}A_{i,aq}&\text{ if $d_i = d - 1$,} 
\\ A_{i,aq^{-2}}A_{i,a}A_{i,aq^2}&\text{ if $d_i = d - 2$.} \end{cases}\end{align}
By inspection of these formulas and the formulas for the monomials
$A_{i,a}$, we check that the $U_{i,a}$'s belong to
$\mathbb{Z}[W_{i,a}^{\pm 1}]_{i\in I, a\in\mathbb{C}^\times}$. This
completes the proof.
\end{proof}

In particular, we obtain the following result.

\begin{prop}\label{wextr} For any dominant monomial $m\in
  \mathbb{Z}[W_{i,a}]_{i\in I, a\in\mathbb{C}^\times}$,
the extremal $\ell$-weights of $L(m)$ belong to $\mathbb{Z}[W_{i,a}^{\pm 1}]_{i\in I, a\in\mathbb{C}^\times}$.
\end{prop}

It implies the following result which will also be useful for our purposes.

\begin{cor}\label{wextr2} Let $i,k\in I$ so that $C_{k,i} < -1$. For any $w\in W$, 
$Y_{w(\omega_i),1}$ is a Laurent monomial in the $W_{j,b}$, $j\in I$, $b\in\mathbb{C}^\times$.
\end{cor}

\begin{proof} This follows from Proposition \ref{wextr} as $d_i$ is
  equal to the lacing number $d$ in this case, and so $Y_{i,1} = W_{i,1}$. 
\end{proof}

\section{Extended $TQ$-relations}    \label{TQ}

The plan of this section is as follows. In Section \ref{rec} we will
recall the generalized Baxter $TQ$-relations in the Grothendieck ring
of the category $\mathcal{O}$ that we established in \cite{FH}
(Theorem \ref{classtq}) and their $q$-character versions. Next, in
Section \ref{genChari} we will extend the Chari braid group action
\cite{C} from the ring $\Yim$ (see Section \ref{secextm} above) to a
larger ring $\Yim' \supset \Yim$. Then in Section \ref{esys} we will
state the Extended $TQ$-relations Conjecture \ref{exttq}. 
  Extended $TQ$-relations are labeled by elements of the Weyl group
  $W$ of $\g$. The relation corresponding to $w \in W$ involves
  classes of simple representations from the category $\OO$ whose
  highest $\ell$-weights are labeled by the weights $w(\omega_i), i \in
  I$, and are obtained using the generalized Chari's action which we
construct in Section \ref{genChari}. In Section
  \ref{simplerefl} we will consider in more detail the case $w=s_i$. In
  Section \ref{wga} we will motivate the Weyl group action introduced
  in \cite{FH3} by the results of Section
  \ref{simplerefl}.  Finally, in Section \ref{add} we will introduce
some homomorphisms that we will use in the next section.

\subsection{Recollections}    \label{rec}

We recall the $TQ$-relations which were conjectured in
\cite{Fre} and proved in \cite{FH}. Recall the ring $\Yim =
\mathbb{Z}[Y_{i,a}^{\pm 1}]_{i\in I, a\in\mathbb{C}^\times}$.

Denote by $\wt{K}_0(\OO)$ the localization of the ring
  $K_0(\OO)$ with respect to the elements $[L(\Psib_{i,a})]$, $i \in I,
  a \in \C^\times$. Define a homomorphism
  \begin{equation}    \label{loc1}
  {\mc B}: \Yim \to \wt{K}_0(\OO)
  \end{equation}
  by the formula
  \begin{equation}    \label{loc11}
  Y_{i,a} \mapsto [\omega_i]
  \dfrac{[L(\Psib_{i,aq_i^{-1}})]}{[L(\Psib_{i,aq_i})]}.
  \end{equation}

\begin{thm}\label{classtq}\cite{FH}    \label{thmFH}
Let $V$ be a finite-dimensional simple representation of
$U_q(\wh\g)$ and $\chi_q(V) \in \Yim$ its
$q$-character. Then ${\mc B}(\chi_q(V)) = [V]$ in
$\wt{K}_0(\OO)$.
\end{thm}

By clearing the denominators in the relation ${\mc
  B}(\chi_q(V)) = [V]$ we obtain an algebraic relation in
$K_0(\OO)$. This is the $TQ$-relation in $K_0(\OO)$.

The following is an equivalent statement. Let $\wt{\mc E}_\ell$ be the
localization of the ring ${\mc E}_\ell$ with respect to the elements
$\chi_q(L(\Psib_{i,a}))$, $i \in I, a \in \C^\times$. Define a
homomorphism
  \begin{equation}    \label{loc1c}
  {\mc B}': \Yim \to \wt{\mc E}_\ell
  \end{equation}
  by the formula
  \begin{equation}    \label{loc1c1}
  Y_{i,a} \mapsto [\omega_i]
    \dfrac{\chi_q(L(\Psib_{i,aq_i^{-1}}))}{\chi_q(L(\Psib_{i,aq_i}))}.
  \end{equation}
  Recall that $\Yim$ is a subring of ${\mc E}_\ell$. Hence it is a
  subring of $\wt{\mc E}_\ell$.
  
\begin{thm}\cite{FH}    \label{corTQ}
Let $V$ be a finite-dimensional simple representation of
$U_q(\wh\g)$ and $\chi_q(V) \in \Yim$ its
$q$-character. Then we have ${\mc B}'(\chi_q(V)) = \chi_q(V)$ in
$\wt{\mc E}_\ell$.
\end{thm}

By clearing the denominators in the relation ${\mc B}'(\chi_q(V)) =
\chi_q(V)$ we obtain an algebraic relation in ${\mc E}_\ell$. This is
the $TQ$-relation expressed in terms of $q$-characters.

For readers convenience, we recall the proof of these two theorems
given in \cite{FH}.

\medskip

\noindent{{\em Proof of Theorems \ref{classtq} and \ref{corTQ}.}}
Recall from Theorem \ref{formuachar} that 
\begin{equation}    \label{pref}
    \chi(L(\Psib_{i,a})) = \chi_q(L_{i,a}^+) =
    \Psib_{i,a} \chi(L_{i,a}^+).
\end{equation}
    Therefore the relation of Theorem \ref{corTQ} can be described as
    follows: we replace each $Y_{i,a}, i \in I$,
      appearing in $\chi_q(V)$ with
$[\omega_i] \chi_q(L(\Psib_{i,aq_i^{-1}}))(\chi_q(L(\Psib_{i,aq_i})))^{-1}$
    and equate the result with $\chi_q(V)$. But according to formula
    \eqref{pref}, this replacement is equivalent to replacing each
    $Y_{i,a}, i \in I$, with $[\omega_i]
    \Psib_{i,aq_i^{-1}}\Psib_{i,aq_i}^{-1}$. According to
    Definition \ref{defY}, the latter ratio is equal to $Y_{i,a}$, so
    this replacement tautologically gives $\chi_q(V)$. Thus, we
    obtain the statement
    of Theorem \ref{corTQ}. Recalling the injectivity of the
    $q$-character homomorphism (Proposition \ref{inject}), we obtain
    the $TQ$-relations of Theorem \ref{classtq}.\qed

\begin{rem}
We had noted in \cite[Remark 4.10]{FH} that there is an analogous
$TQ$-relation in terms of the simple representations
$L(\Psib_{i,a}^{-1})$. Both relations appear as particular cases
of our extended $TQ$-relations below, corresponding to
  the identity element and the longest element of the Weyl group,
  respectively.
\end{rem}

\subsection{Generalization of Chari's braid group
  action}    \label{genChari}

Recall the Chari braid group action on the ring $\Yim$, which we
described in Section \ref{secextm}. Introduce the following extension
$\mathcal{Y}'$ of the ring $\Yim$:
\begin{equation}    \label{Yprime}
\mathcal{Y}' :=
\mathbb{Z}[\Psib_{i,a}^{\pm 1}]_{i\in I,a\in\mathbb{C}^\times} \otimes_{\Z} \Z(P)
= \mathbb{Z}[\Psib_{i,a}^{\pm 1}, y_j^{\pm 1}]_{i\in I,a\in\mathbb{C}^\times;j \in I}
\supset \mathcal{Y},
\end{equation}
where we have used formula \eqref{ZP}.

We are going to define ring automorphisms $T'_i: \Yim' \to \Yim', i
\in I$, satisfying the braid group relations, whose restrictions to
$\Yim \subset \Yim'$ are the Chari braid group operators $T_i, i \in I$,
on $\Yim$ given by formula \eqref{TiYj}. Thus, we will obtain an
extension of the Chari braid group action from $\Yim$ to $\Yim'$.

Recall the elements $\wt{\Psib}_{i,a} \in \Yim', i \in I$,
which we introduced in \cite{FH2} :
  \begin{multline}\label{psfh2}\wt{\Psib}_{i,a} :=
    \\ \Psib_{i,a}^{-1} \left(\prod_{j\in I,C_{i,j} =
      -1}\Psib_{j,aq_i}\right) \left(\prod_{j\in I, C_{i,j} =
  -2}\Psib_{j,a}\Psib_{j,aq^2}\right)\left(\prod_{j\in I, C_{i,j} =
      -3}\Psib_{j,aq^{-1}}\Psib_{j,aq}\Psib_{j,aq^3}\right).
  \end{multline}

Define a ring automorphism $\sigma$ of $\Yim'$ by the formula
\begin{equation}    \label{sigm}
\sigma(\Psib_{i,a}) = \Psib_{i,a^{-1}}^{-1}, \qquad i\in I, \quad
a\in\mathbb{C}^\times; \qquad \sigma([\omega]) = [\omega], \quad
\omega \in P.
\end{equation}

Define a ring automorphism $T'_i: \Yim' \to \Yim', i \in I$, by the
formulas
\begin{multline}    \label{TiPj}
T'_i(\Psib_{i,a}) = \sigma(\wt{\Psib}_{i,a^{-1}q_i^{-2}}^{-1}) =
\\ \Psib_{i,aq_i^2}^{-1} \left(\prod_{j\in I,C_{i,j} =
      -1}\Psib_{j,aq_i}\right) \left(\prod_{j\in I, C_{i,j} =
  -2}\Psib_{j,a}\Psib_{j,aq^2}\right)\left(\prod_{j\in I, C_{i,j} =
  -3}\Psib_{j,aq^{-1}}\Psib_{j,aq}\Psib_{j,aq^3}\right).
 \end{multline}
\begin{equation}
T'_i(\Psib_{j,a}) = \Psib_{j,a}, \quad j \neq i; \qquad
T'_i[\omega] = [s_i(\omega)] , \quad \omega \in P.
\end{equation}

This is indeed an automorphism, with the inverse given by 
$$(T'_i)^{-1}(\Psib_{i,a}) = \wt{\Psib}_{i,aq_i^{-2}}, \qquad
(T'_i)^{-1}(\Psib_{j,a}) = \Psib_{j,a}, \quad j \neq i; \qquad
(T'_i)^{-1}[\omega] = [s_i(\omega)] , \quad \omega \in P.$$

\begin{thm}
(1)  The operators $T'_i, i \in i$, generate an action of the braid group
  associated to $\g$ on $\Yim'$.

(2) The restriction of $T'_i$ to $\Yim \subset
\Yim'$ is equal to the Chari operator $T_i$ (see formula \eqref{TiYj}).
\end{thm}

\begin{proof} It suffices to prove that the braid relations are
satisfied on the generators of $\Yim'$. Since the $T_i, i \in I$,
acting on $\Yim$, satisfy the braid relations by Chari's theorem
\cite{C}, we know that the $T_i', i \in I$, satisfy the braid
relations on the generators $Y_{j,a}, j \in I$, of $\Yim \subset
\Yim'$. This is also clear for the generators $y_j$, $j\in I$, of
$\Yim'$ because $T_i'$ acts on $\Z[y_j^{\pm 1}]_{j \in I} = \Z(P)$ in
the same way as $s_i \in W$ acts on $\Z(P)$.

Hence the braid relations $\mathcal{B}_{i,k}, i,k \in I$, are
satisfied on the monomials $\Psib_{j,aq_j^{-1}}\Psib_{j,aq_j}^{-1} =
Y_{j,a}[-\omega_i], j \in I$. It remains to show that they are also
satisfied on $\Psib_{j,a}^{\pm 1}, j \in I$.

Consider first the simplest non-trivial case: $C_{i,k} = C_{k,i} =
-1$. In this case $\mathcal{B}_{i,k}$ is the relation
$(T_k')^{-1}(T_i')^{-1}(T_k')^{-1}T_i' T_k' T_i' = \on{Id}$.
Since it is satisfied on $\Psib_{j,aq_j^{-1}}\Psib_{j,aq_j}^{-1}$, we have
$$(T_k')^{-1}(T_i')^{-1}(T_k')^{-1}T_i' T_k'
T_i'(\Psib_{j,aq_j^{-1}}\Psib_{j,aq_j}^{-1}) =
\Psib_{j,aq_j^{-1}}\Psib_{j,aq_j}^{-1}.$$
Therefore,
$$\Psib_{j,aq_j^{-1}}^{-1}(T_k')^{-1}(T_i')^{-1}(T_k')^{-1}T_i' T_k'
T_i'(\Psib_{j,aq_j^{-1}}) =
\Psib_{j,aq_j}^{-1}(T_k')^{-1}(T_i')^{-1}(T_k')^{-1}T_i' T_k' T_i'(\Psib_{j,aq_j^{-1}}).$$ 
This implies that
$\Psib_{j,aq_j^{2m}}^{-1}(T_k')^{-1}(T_i')^{-1}(T_k')^{-1}T_i' T_k'
T_i'(\Psib_{j,aq_j^{2m}})$ does not depend on $m \in \Z$.
But this expression is a Laurent monomial in
  $\Psib_{l,aq^r}^{\pm 1}, l \in
I, r \in \Z$. Therefore, it must be equal to $1$. Hence we obtain
that the braid relation
$\mathcal{B}_{i,k}$ holds on
$\Psib_{j,a}$ in this case. The braid relations for other values of $C_{i,k}$ and
$C_{k,i}$ are proved in a similar way.
\end{proof}

Recall the monomials $Y_{w(\omega_i),a} \in \Yim, w \in W,
  i \in I, a \in \C^\times$ introduced in Section \ref{secextm} (see
  equation \eqref{Ywia}):
$$
Y_{w(\omega_i),a} := T_w(Y_{i,a}).
$$
Here $T_w$ is the Chari operator on $\Yim$ corresponding to
$w$, viewed as an element of the braid group via its reduced
decomposition. According to Corollary \ref{wextr2}, if
  $C_{k,i} < - 1$, then $Y_{w(\omega_i),a}$ is a Laurent polynomial in
$W_{k,b}$.

Now we compute the $\ell$-weights $\sigma \circ T'_w \circ
\sigma(\Psib_{i,a}), w \in W$, where $T'_w$ is obtained from a reduced
decomposition of $w$ viewed as an element of the braid group and
$\sigma$ is the ring automorphism of $\Yim'$ introduced in (\ref{sigm}).

\begin{defi}\label{subs}
For $i\in I$, $a\in\mathbb{C}^\times$, and $w\in W$, define an
$\ell$-weight $\Psib_{w(\omega_i),a}$ by the following formulas.

The $\ell$-weight $\Psib_{w(\omega_i),1}$ is defined from the
factorization of $Y_{w(\omega_i),1}$ as a product of the variables 
$Y_{k,b}^{\pm 1}$ by replacing

$Y_{k,b}$ by $\Psib_{k,b^{-1}}$ if $d_i = d_k$,

$Y_{k,b}$ by $\Psib_{k,b^{-1}q}\Psib_{k,b^{-1}q^{-1}}$ if $C_{i,k} = -2$,

$Y_{k,b}$ by
  $\Psib_{k,b^{-1}q^2}\Psib_{k,b^{-1}}\Psib_{k,b^{-1}q^{-2}}$ if
$C_{i,k} = -3$,

$W_{k,b}$ by $\Psib_{k,b^{-1}}$ if $C_{k,i} < - 1$.

\noindent Finally, we set 
$$\Psib_{w(\omega_i),a} := \tau_a(\Psib_{w(\omega_i),1}).$$
\end{defi}

\begin{prop}\label{proptprime}
We have
$$\Psib_{w(\omega_i),a} = \sigma \circ T'_w \circ
\sigma(\Psib_{i,a}), \qquad w \in W,
$$
where $T'_w = T'_{i_1} T'_{i_2} \ldots T'_{i_k}$
for any reduced decomposition $w = s_{i_1} \ldots s_{i_k}$ and
$\sigma$ is given by formula \eqref{sigm}.
\end{prop}

This statement is a consequence of Proposition \ref{tctw} below and
its proof is given at the end of Section \ref{comptheta}.

\begin{rem}    \label{orbit}
    If $\g$ is simply-laced (i.e. $d_i=1$ for all $i \in I$), then we
    can obtain $\Psib_{w(\omega_i),1}$ for all $i \in I$ by
    replacing $Y_{k,b}$ by $\Psib_{k,b^{-1}}$ in $Y_{w(\omega_i),1}$ for
    all $k$ and $b$. However, if $\g$ is not simply-laced, then this
    is only true for
    those $k$ for which the lengths of $\al_i$ and $\al_k$ are the
    same. If the length of $\al_i$ is shorter than the length of
    $\al_k$, we replace $Y_{k,b}$ by a product of two or three
    $\Psi$'s (depending on the lacing number of $\g$). And if the
    length of $\al_i$ is longer than the length of $\al_k$, then we
    have to
    replace $W_{k,b}$, which is the product of two or three $Y$'s
    (depending on the lacing number of $\g$), by $\Psib_{k,b^{-1}}$
    (this procedure is well-defined by Corollary \ref{wextr2}). Under
    this construction, the Cartan matrix gets replaced by its
    transpose, which we can see as a non-trivial manifestation of
    Langlands duality.
  \end{rem}

\begin{example}    \label{exPsi}
  (i) For $\g=\sw_2$ we have
$$Y_{\omega_1,1} = Y_{1,1}\text{ and }  Y_{w_0(\omega_1),1} = Y_{1,q^2}^{-1}$$
and so 
$$\Psib_{\omega_1,a} = \Psib_{1,a}\text{ and } \Psib_{w_0(\omega_1),a}
= \Psib_{1,aq^{-2}}^{-1}.$$

(ii) It is known from \cite{Fre2} that $Y_{w_0(\omega_i),1} = Y_{\overline{i},q^{ d h^\vee}}^{-1}$. Hence
$$\Psib_{w_0(\omega_i),a} = \Psib_{\overline{i},aq^{- d h^\vee}}^{-1}.$$
\end{example}

\subsection{Extended $TQ$-relations Conjecture}\label{esys}

We now propose the conjectural extended
$TQ$-relations labeled by elements $w$ of the Weyl group $W$, so
that the original $TQ$-relations of Theorem \ref{classtq} correspond
to the identity element of $W$.

Recall that for any monomial $\Psib$ in
  $\Psib_{i,b}^{\pm 1}, i \in I, b \in \C^\times$ (such as
  $\Psib_{w(\omega_i),a}$), there is a unique simple highest
  $\ell$-weight module $L(\Psib)$ in the category $\OO$ (see Lemma
  \ref{simple}).

Fix $w \in W$, and denote by $\wt{K}^w_0(\OO)$ the
  localization of the ring
  $K_0(\OO)$ with respect to the elements $[L(\Psib_{w(\omega_i),a})]$, $i \in I,
  a \in \C^\times$. Define a homomorphism
  \begin{equation}    \label{locw}
  {\mc B}_w: \Yim \to \wt{K}^w_0(\OO)
  \end{equation}
  by the formula
  \begin{equation}    \label{locw1}
  Y_{i,a} \mapsto [w(\omega_i)]
  \dfrac{[L(\Psib_{w(\omega_i),aq_i^{-1}})]}{[L(\Psib_{w(\omega_i),aq_i})]}.
  \end{equation}
Note that $\wt{K}^e_0(\OO) = \wt{K}_0(\OO)$ introduced in Section
\ref{rec} and ${\mc B}_e$ is ${\mc B}$ given by formulas \eqref{loc1}
and \eqref{loc11}.

\begin{conj}\label{exttq}
  Let $V$ be a finite-dimensional simple representation of
$U_q(\wh\g)$ and $\chi_q(V) \in \Yim$ its
$q$-character. Then ${\mc B}_w(\chi_q(V)) = [V]$ in
$\wt{K}^w_0(\OO)$.
\end{conj}

By clearing the denominators in the relation ${\mc B}_w(\chi_q(V)) =
[V]$ we obtain an algebraic relation in $K_0(\OO)$. This is the
conjectural extended $TQ$-relation in $K_0(\OO)$ corresponding to $w
\in W$.

Informally, Conjecture \ref{exttq} can be interpreted as
  follows: let us make the following substitution for every variable
  $Y_{i,a}, i \in I$, appearing in the $q$-character $\chi_q(V)$:
$$Y_{i,a} \mapsto [w(\omega_i)]
  \dfrac{[L(\Psib_{w(\omega_i),aq_i^{-1}})]}{[L(\Psib_{w(\omega_i),aq_i})]}.
  $$
Equating the resulting expression with $[V]$, we obtain the extended
$TQ$-relation in $\wt{K}^w_0(\OO)$ corresponding to $w \in W$ (and by
clearing the denominators, in $K_0(\OO)$).

%To state the extended $QQ$-system Conjecture, we need the following definition. Recall Lemma \ref{wn} (1). There is a unique family 
%$$\chi_{i,w} \in \ZZ[(1 - [-\alpha])^{-1}]_{\alpha \in \Delta}\text{ for }w\in W, i\in I$$ 
%so that $\chi_{i,e} = 1$ and 
%\begin{equation}\label{indchii}\chi_{i,ws_i}\chi_{i,w} = (1 - [- w(\alpha_i)])^{-1} \prod_{j\neq i}\chi_{j,w}^{-C_{i,j}}\text{ if $l(w s_i) = l(w) + 1$.}\end{equation}
%The existence and uniqueness is proved by induction on the length of $w$.

%\begin{example} For $i\in I$, we have
%$$\chi_{i,s_i} = (1 - [-\alpha_i])^{-1}.$$ 
%\end{example}

Since the $q$-character homomorphism is injective, Conjecture
\ref{exttq} is equivalent to the following statement. Let
  $\wt{\mc E}^w_\ell$ be the localization of the ring ${\mc E}_\ell$
  with respect to the elements $\chi_q(L(\Psib_{w(\omega_i),a}))$, $i \in I, a
  \in \C^\times$. Define a homomorphism
  \begin{equation}    \label{locwc}
  {\mc B}'_w: \Yim \to \wt{\mc E}^w_\ell
  \end{equation}
  by the formula
  \begin{equation}    \label{locwc1}
  Y_{i,a} \mapsto [w(\omega_i)]
    \dfrac{\chi_q(L(\Psib_{w(\omega_i),aq_i^{-1}}))}{\chi_q(L(\Psib_{w(\omega_i),aq_i}))}.
  \end{equation}
Note that $\wt{\mc E}^e_\ell = \wt{\mc E}_\ell$ introduced in Section
\ref{rec} and ${\mc B}'_e$ is ${\mc B}'$ given by formulas \eqref{loc1c}
and \eqref{loc1c1}.
  
  Recall that $\Yim$ is a subring of ${\mc E}_\ell$ and hence of
  $\wt{\mc E}^w_\ell$.

\begin{conj}    \label{exttq1}
Let $V$ be a finite-dimensional simple representation of
$U_q(\wh\g)$ and $\chi_q(V) \in \Yim$ its
$q$-character. Then we have ${\mc B}'_w(\chi_q(V)) = \chi_q(V)$ in
$\wt{\mc E}^w_\ell$.
\end{conj}

By clearing the denominators in the relation ${\mc B}'_w(\chi_q(V)) =
\chi_q(V)$ we obtain an algebraic relation in ${\mc E}_\ell$. This is
the extended $TQ$-relation corresponding to $w \in W$ expressed in
terms of $q$-characters.

Informally, this statement can be interpreted as
  follows: in $\wt{\mc E}^w_\ell$, the $q$-character $\chi_q(V)$
    is equal to the expression obtained by replacing every
    $Y_{i,a}, i \in I$, appearing in $\chi_q(V)$ with
$$[w(\omega_i)]
\dfrac{\chi_q(L(\Psib_{w(\omega_i),aq_i^{-1}}))}{\chi_q(L(\Psib_{w(\omega_i),aq_i}))}$$
(by clearing the denominators, we obtain a relation in
${\mc E}_\ell$).

\subsection{The case of simple reflections}    \label{simplerefl}

For $w=e$, Conjecture \ref{exttq1} is proved in Theorem
\ref{corTQ}. In this case, we have the $q$-character formula which
follows from Theorem \ref{formuachar}:
    $$
    \chi_q(L(\Psib_{i,a})) =
    \left[\Psib_{i,a}\right] \chi(L_{i,a}^+).$$

    Next, consider the case when $w$ is a simple reflection, $w=s_i, i
    \in I$. In this case, we have $\Psib_{s_i(\omega_i),a} =
    \wt{\Psib}_{i,aq_i^{-2}}$ (see Example \ref{exPsi},(ii)), and the
    $q$-character formula for $L(\Psib_{s_i(\omega_i),a})$ is given in
    the following theorem.

  \begin{thm}\label{fex}\cite{FH2, FHR}
    The $q$-character of
  $L(\Psib_{s_i(\omega_i),a}) = L(\wt{\Psib}_{i,aq_i^{-2}}))$ is equal to
  $$\chi_q(L(\Psib_{s_i(\omega_i),a})) =
    [\wt{\Psib}_{i,aq_i^{-2}}] \; \chi(L(\wt{\Psib}_{i,aq_i^{-2}})))
    \; (1 - [-\alpha_i]) \;  \Sigma^e_{i,aq_i^{-2}}
  $$
 where
\begin{equation}    \label{Sigmai}
 \Sigma^e_{i,a}:= \sum_{k\geq 0} \; \prod_{0 < j \leq k}
  A_{i,aq_i^{-2j+2}}^{-1} = 1 + A_{i,a}^{-1} (1 + A_{i,aq_i^{-2}}^{-1}(1 +
  \ldots )).
\end{equation}
\end{thm}

\subsection{Weyl group action}    \label{wga}

 Conjecture \ref{exttq} and Theorem \ref{fex} suggest
  defining an operator $\Theta_i$ sending $Y_{i,a}$ to
\begin{equation}    \label{sioi}
\Theta_i(Y_{i,a}) :=
[s_i(\omega_i)]
\frac{\chi_q(L(\wt{\Psib}_{i,aq_i^{-3}}))}{\chi_q(L(\wt{\Psib}_{i,aq_i^{-1}}))}
= Y_{i,a}A_{i,aq_i^{-1}}\frac{\Sigma_{i,aq_i^{-3}}^e}{\Sigma_{i,aq_i^{-1}}^e}
\end{equation}
and posing the question whether these operators satisfy the relations
of the Weyl group $W$.

  To make sense of this statement, we need to deal with the fact that
  $\Sigma^e_{i,a}$ lies in a completion of
\begin{equation}    \label{Yim}
  \Yim = \ZZ[Y_{i,a}^{\pm 1}]_{i\in I,a\in\CC^\times}
\end{equation}
  (and so in particular $\Theta_i$
  does not preserve $\Yim$). In \cite{FH3} we resolved this problem by
  extending the operators $\Theta_i$ to a direct sum of completions of
  $\Yim$ labeled by $w \in W$.

 The key point here is that $\Sigma_{i,a}^e$ is a solution of the
 $q^{2d_i}$-difference equation
\begin{equation}\label{siae}\Sigma_{i,a} = 1 +
  A_{i,a}^{-1}\Sigma_{i,aq^{-2d_i}}.
\end{equation}
In order to define the operators $\Theta_i$ we also need another
solution of this equation; namely,
\begin{equation}\label{Sigmasi}
  \Sigma_{i,a}^{s_i} := - A_{i,aq_i^2}(1 + A_{i,aq_i^4}(1 +
  \cdots)\cdots).
\end{equation}

Using these two solutions, we defined in \cite{FH3} operators
$\Theta_i, i \in I$, acting on the direct sum
$$\Pi := \bigoplus_{w\in W} \wt{\mathcal{Y}}^w$$ of topological
completions $\wt{\mathcal{Y}}^w$ of $\mathcal{Y}$ with respect to a
partial ordering $\prec_w$ on the multiplicative
group ${\mc M}$ of monomials in $\Yim$. It is defined as follows.

First, define the partial ordering $<_w$ on $P$ by the rule $\omega
<_w \omega'$ is and only $w(\omega) < w(\omega')$, where $<$ denotes
the standard partial ordering (see Section \ref{Liealg}). Next, recall the
multiplicative group ${\mc M}$ of monomials in the variables
$Y_{i,a}^{\pm 1}$ and the homomorphism $\varphi: {\mc M} \to P$
introduced in Section \ref{qchar}. Taking the pull-back of the partial
ordering $<_w$ under $\varphi$, we obtain a partial ordering $\prec_w$
on ${\mc M}$. Finally, define a topology on the ring $\mathcal{Y}$
invariant under addition, in which the base of open neighborhoods of the
zero element of $\mathcal{Y}$ consists of the sets $U_M, M \in {\mc
  M}$, of all finite linear combinations of monomials $m_i \in {\mc
  M}$ satisfying $m_i \prec_w M$. By definition, $\wt{\mathcal{Y}}^w$ is
the completion of $\mathcal{Y}$ with respect to this topology. Since $m_1
\prec_w M_1$ and $m_2 \prec_w M_2$ implies $m_1 m_2 \prec_w M_1 M_2$,
the ring structure on $\mathcal{Y}$ naturally extends to
$\wt{\mathcal{Y}}^w$. Thus, $\wt{\mathcal{Y}}^w$ is a complete
topological ring.

The algebra $\mathcal{Y}$ embeds diagonally in 
$\Pi$ and we have canonical projections $E_w : \Pi\rightarrow
\wt{\mathcal{Y}}^w$, $i\in I$.

\begin{defi}\cite{FH3}
For each $i \in I$, introduce the operator $\Theta_i: \Pi \to \Pi$,
$$\Theta_i = (\Theta_i^{v} : \wt{\mathcal{Y}}^{v}\rightarrow
\wt{\mathcal{Y}}^{v s_i})_{v\in W}$$
with $\Theta_i^v, v \in W$, given by the formulas
$\Theta^v_i(Y_{j,a}) = Y_{j,a}$ if $j\neq i$ and
\begin{equation}\label{defti}
  \Theta^v_i(Y_{i,a}) =
  Y_{i,a}A_{i,aq^{-d_i}}^{-1}\frac{\Sigma^{vs_i}_{i,aq^{-3d_i}}}{\Sigma^{vs_i}_{i,aq^{-d_i}}}
\end{equation}
where $\Sigma^{vs_i}_{i,a}\in \wt{\mathcal{Y}}^{vs_i}$ is one of the
two solutions of the $q^{2d_i}$-difference equation (\ref{siae}):
\begin{equation}    \label{twosol}
  \Sigma^{vs_i}_{i,a} := \begin{cases} \Sigma^e_{i,a} & \; \text{if}
    \quad l(vs_i)<l(v), \\ \Sigma^{s_i}_{i,a} & \; \text{if}
    \quad l(vs_i)>l(v). \end{cases}
\end{equation}
\end{defi}

Recall that $\Sigma^e_{i,a}$ is given by formula \eqref{Sigmai} and
$\Sigma^{s_i}_{i,a}$ is given by formula \eqref{Sigmasi}.

\begin{thm}\label{mfh}\cite{FH3}
  (1)There is an action of the Weyl group $W$ of $\g$ on $\Pi$, with
    the simple reflection $s_i \in W (i \in I)$ acting as $\Theta_i$.

  (2) The subring $\mathcal{Y}^W$ of
  $W$-invariants in $\Yim$ is equal to the image of the $q$-character
  homomorphism $\chi_q: K_0({\mc C}) \to \Yim$, and hence
  $\mathcal{Y}^W$ is isomorphic to $K_0({\mc
    C})$ (here ${\mc C}$ is the category of type $1$ finite-dimensional
representations of $U_q(\wh{\Glie})$).
\end{thm}

Theorem \ref{mfh},(1) allows us to define an operator $\Theta_w : \Pi
\rightarrow \Pi$ for an arbitrary element $w\in W$ using a reduced
decomposition of $w$ in terms of the simple reflections $s_i$. This
operator acts on $\Pi$ as follows: $\Theta_w = (\Theta_w^{v}:
\wt{\mathcal{Y}}^{v}\rightarrow \wt{\mathcal{Y}}^{vw^{-1}})_{v\in W}$.

%\begin{rem}
%By Theorem \ref{fex}, formula \eqref{defti} expresses the ratio of
%$q$-characters of modules appearing in Conjecture \ref{exttq} in the
%case $w=s_i$. The invariance of $q$-characters under $\Theta_i$ (see
%Theorem \ref{mfh}) is closely related to the statement of this
%conjecture for simple reflexions. We will make this more precise in
%Example \ref{mainm}.
%\end{rem}

\subsection{Some useful homomorphisms}   \label{add}

Let us recall the homomorphisms $\varpi$ from \cite[Section 2.5]{FH3}
and $\Lambda$ from \cite[Section 6.4]{FH3} which we will use below.

Recall also the partial ordering $<_w$ on $P$ introduced
  above in Section \ref{wga}. Since
\begin{equation}    \label{ZP}
  \Z(P) = \mathbb{Z}[y_i^{\pm 1}]_{i\in I}
\end{equation}
where the variable $y_i$ correspond to the
fundamental weight $\omega_i \in P$, we obtain a partial ordering on
the set of monomials in $\mathbb{Z}[y_i^{\pm 1}]_{i\in I}$ that we
will also denote by $<_w$. We define the topological completion
  $(\mathbb{Z}[y_i^{\pm 1}]_{i\in I})^w$ of $\mathbb{Z}[y_i^{\pm
    1}]_{i\in I}$ with respect to this partial ordering in
the same way in which we defined the completion $\wt{\mathcal{Y}}^w$
of $\mathcal{Y}$ in Section \ref{wga} (with respect to the partial
  ordering $\prec_w$).

Consider the ring
$$\pi := \bigoplus_{w\in
    W}(\mathbb{Z}[y_i^{\pm 1}]_{i\in I})^w.$$
with the diagonal embedding $\mathbb{Z}[y_i^{\pm 1}] \hookrightarrow \pi$.
The assignment $\varpi_{w}(Y_{i,a}) = y_i$ extends to a ring homomorphism
$$\varpi_{w} : \wt{\mathcal{Y}}^w \rightarrow (\mathbb{Z}[y_i^{\pm
    1}]_{i\in I})^w.$$
Thus, we obtain a map
$$\varpi = (\varpi_w)_{w\in W} : \Pi \to \pi.$$
The standard action of the simple reflection $s_i, i \in I$, on ${\mathbb
  Z}[y_j^{\pm 1}]_{j \in I}$, naturally
extends to a collection of $s_i^w : (\mathbb{Z}[y_j^{\pm 1}]_{j\in I})^w\rightarrow
(\mathbb{Z}[y_j^{\pm 1}]_{j\in I})^{ws_i}$ and we obtain a
well-defined action of the Weyl group $W$ on $\pi$.

\begin{lem}\cite{FH3}
The projection $\varpi: \Pi \to \pi$ intertwines the actions of $W$ on the two rings.
\end{lem}

Now recall the multiplicative group ${\mc M}$ of monomials in the
variables $Y_{i,a}^{\pm 1}$ and the homomorphism $\varphi: {\mc M} \to
P$ introduced in Section \ref{qchar}. Taking the pull-back of the
partial ordering $<_w$ under $\varphi$, we obtain a partial ordering
$\prec_w$ on ${\mc M}$ which we used in Section \ref{wga} to define
the completion $\wt{\mathcal{Y}}^w$ of ${\mc Y}$.

A subgroup $\overline{\mathcal{M}}$ of the group of invertible
elements in $\Pi$ is introduced \cite{FH3}. For our purposes, it is
enough to recall that $\overline{\mathcal{M}}$ contains $\mathcal{M}$
and that it is stable by the Weyl group action. We have the truncation
homomorphism
$$\Lambda :
\overline{\mathcal{M}}\rightarrow \mathcal{M}$$ 
which assigns to $P \in\overline{\mathcal{M}}$ the leading monomial of
$E_e(P)$.

 Recall the ring automorphism $\sigma$ of $\Yim'$ given by
  formula \eqref{sigm}. We will denote by the same symbol its
  restrictions to $\Yim \subset \Yim'$ and to $\mathcal{M} \subset
  \Yim$ (the latter is a group automorphism). Both are defined by the
  formula $\sigma(Y_{i,a}) = Y_{i,a^{-1}}$ for $i\in I$ and
  $a\in\mathbb{C}^\times$.

We proved in \cite[Section 6.4]{FH3} that for any $i\in I$, the Chari
operator $T_i$ of section \ref{secextm} is also equal to
\begin{equation}\label{til}T_i = \sigma \circ \Lambda \circ \Theta_i \circ \sigma.\end{equation} 

This will be generalized below for $w\in W$ by using homomorphism  
$$\Lambda_w : \overline{\mathcal{M}}\rightarrow \mathcal{M}.$$ 
It is the leading term, relatively to the partial ordering $\prec_w$, 
of the expansion in $\wt{\mathcal{Y}}^w$ of an element of
$\overline{\mathcal{M}}$.
For $w = e$, we have $\Lambda_e = \Lambda$.

\begin{example} For $i\in I$ and $a\in\mathbb{C}^\times$, we have
$$\Lambda_e(\Sigma_{i,a}) = 1\text{ and }\Lambda_{s_i}(\Sigma_{i,a}) = -A_{i,aq_i^2}.$$
\end{example}

For $w\in W$, we will also consider the group homomorphism: 
$$T_w = \Lambda_w \circ \Theta_w  : \mathcal{M}\rightarrow \mathcal{M}.$$

\subsection{Extension of $\Pi$}    \label{extact}
We define the following extension of $\Pi$: 
$$\Pi' := \mathcal{Y}'\otimes_\mathcal{Y} \Pi
$$
where ${\mc Y}'$ is defined by formula \eqref{Yprime}.
Thus,
$$
\Pi' = \bigoplus_{w\in W} \wt{\mathcal{Y}}'{}^w
$$
where
\begin{equation}    \label{Yprimew}
\wt{\mathcal{Y}}'{}^w := \mathcal{Y}'\otimes_\mathcal{Y} \wt{\mathcal{Y}}^w
\end{equation}
is the completion of $\mathcal{Y}'$ consisting of finite sums of
elements of the form $P(\Psib_{i,a}^{\pm 1},y_j^{\pm 1}) \cdot
R(Y_{i,a}^{\pm 1})$, where $P$ is a Laurent polynomial and $R$ is an element
of $\wt{\mathcal{Y}}'{}^w$ which, as we explained in Section \ref{wga},
is the completion of ${\mc Y}$ with respect to the partial ordering
$\prec_w$ defined in Section \ref{add}.

Now we define another completion $\ol{\mathcal{Y}}'{}^w$ as follows:
it consists of finite sums of elements of the form $P(\Psib_{i,a}^{\pm
  1}) \cdot R(Y_{i,a}^{\pm 1},y_j^{\pm 1})$, where $P$ is a Laurent polynomial
and $R$ is an element of the completion of ${\mc Y}$ with respect to
the partial ordering obtained by combining the orderings $\prec_w$ and
$<_w$ defined in Section \ref{wga}. We then set
$$
\ol\Pi' := \bigoplus_{w\in W} \ol{\mathcal{Y}}'{}^w.
$$

Observe that $\ol{\mathcal{Y}}'{}^e$ is naturally realized as a
subring of ${\mc E}_\ell$. The following result follows from \cite{HJ}
and \cite[Theorem 4.19]{W} (see Definition \ref{K0} for the definition
of $K_0(\OO)$).

\begin{prop}
The $q$-character homomorphism $\chi_q: K_0({\mc O}) \to {\mc E}_\ell$
takes values in $\ol{\mathcal{Y}}'{}^e$ and the resulting homomorphism
$K_0({\mc O}) \to \ol{\mathcal{Y}}'{}^e$ is an isomorphism.
\end{prop}

Next, we extend the Weyl group action on $\Pi$ to a braid group action
on $\Pi'$ and $\ol\Pi'$. Note that this action is also discussed in
\cite{GHL}.

\begin{prop}    \label{extact1}
The operators $\Theta_i'$, $i\in I$ given by the formulas
$$\Theta_i'([\omega]) = [s_i(\omega)],
\qquad
\Theta_i'(\Psib_{j,a}) = 
\left\{ 
\begin{array}{lc}
\Psib_{j,a} & \text{ if $j \not = i$}, \\[2mm]
\widetilde{\Psib}_{i,aq_i^{-2}}\Sigma_{i,aq_i^{-2}} & \mbox{ if }j = i.
\end{array}
\right.
$$
generate an action of the braid group associated to $\g$ on $\Pi'$ and
$\ol\Pi'$.
\end{prop}

\begin{proof}
The braid relations are proved from the rank $2$ case in the same way as for the operators $\Theta_i$ in \cite{FH3}.
For example, in type $A_2$, we have for $i\neq j$
$$\Theta_j' \Theta_i'(\Psib_{i,a}) =  \Psib_{j,aq^{-3}}^{-1}\Sigma_{ji,aq^{-2}}$$
which is invariant under $\Theta_i'$. The proof for types $B_2$ and $G_2$ is similar, using \cite[Sections 4.7, 4.8]{FH3}.
\end{proof}

It is clear that the operators $\Theta_i'$ preserve the groups of
invertible elements $(\Pi')^\times$, $(\Pi)^\times$, and
$\tb^\times$. The following result is obtained by a direct computation.

\begin{prop}    \label{actG}
The operators $\Theta_i', i\in I$, preserve the
subgroup $G$ of $(\Pi')^\times$ generated by $(\Pi)^\times$ and the
$\Psib_{j,a}$'s. On this subgroup, we have $(\Theta_i')^2 = \text{Id}$ on
$(\Pi)^\times$ and $(\Theta_i')^2(\Psib_{j,a})$ is equal to
$\Psib_{j,a}$ up to a factor in $\pm \tb^\times$.  Hence we have a
well-defined action of the Weyl group $W$ on $G/(\pm \tb^\times)$.
\end{prop}

\section{Weyl group action and solutions of the extended $TQ$-relations}\label{wgao}

In this section we use the Weyl group action from Section
  \ref{wga} to prove Conjectures \ref{exttq} and \ref{exttq1} when $w
  \in W$ is a simple reflection. We then relate these conjectures for
  arbitrary elements $w \in W$ to a conjectural formula for the
  $q$-character of $L(\Psib_{w(\omega_i),aq_i^{-1}})$, up to
  non-zero factor from ${\mc E} \subset {\mc E}_\ell$ (Conjecture
  \ref{upto}). This formula generalizes the formulas in Theorem
  \ref{fex} in the case when $w$ is a simple reflection. In the
  process, we define generating functions $Q_{w(\omega_i),a}\in \Pi$
  (Theorem \ref{exwo}) and show that they satisfy the extended
  $TQ$-relations (Theorem \ref{wbax}). In the next section, we will
  show that they also satisfy the extended $QQ$-system (Theorem
  \ref{QQrel}).

\subsection{Relations obtained from the Weyl group action}

For each element $w$ of the Weyl group $W$, we have the
  extended $TQ$-relations from Conjecture \ref{exttq} and their
  equivalent $q$-character versions from Conjecture \ref{exttq1}. In
  the special case $w=s_j, j \in I$, Theorem \ref{fex} gives us an
  explicit $q$-character formula for $L(\Psib_{s_j(\omega_j),a})$, and
  therefore we obtain the following equivalent reformulation of
  Conjecture \ref{exttq1} in the case $w=s_j, j \in I$.

\begin{cor}    \label{exttq2}
Let $V$ be a finite-dimensional simple representation of
$U_q(\wh\g)$. Its $q$-character $\chi_q(V)$, viewed as an
element of the completion $\wt\Yim^e$ of $\Yim$, is invariant under
the substitution
$$
Y_{i,a} \mapsto E_e(\Theta_j(Y_{i,a})), \qquad \forall i \in I, a \in
\C^\times.
$$
\end{cor}

We will now state and prove a generalization of Conjecture
\ref{exttq2} for an arbitrary element $w \in W$ (instead of
$w=s_i$). In fact, we will formulate two versions: for $\wt\Yim^e$
and for $\Pi = \oplus_{w \in W} \wt\Yim^w$.

\begin{thm}\label{wbax1}
Let $V$ be a finite-dimensional representation
of $U_q(\wh{\Glie})$ and $w\in W$.

(1) The $q$-character $\chi_q(V)$, viewed as an
element of the completion $\wt\Yim^e$ of $\Yim$, is invariant under
the substitution
$$
Y_{i,a} \mapsto E_e(\Theta_w(Y_{i,a})), \qquad \forall i \in I, a \in
\C^\times.
$$

(2) The $q$-character $\chi_q(V)$, viewed as an element of $\Pi$ under
the diagonal embedding $\Yim \hookrightarrow \Pi$, is invariant under
the substitution
$$
Y_{i,a} \mapsto \Theta_w(Y_{i,a}), \qquad \forall i \in I, a \in
\C^\times.
$$
\end{thm}

\begin{proof} Theorem \ref{mfh} implies that the $q$-character
  $\chi_q(V)$, viewed as an element of $\Pi$ under the diagonal
  embedding $\Yim \hookrightarrow \Pi$, is invariant under all
  operators $\Theta_w, w \in W$. This proves part (2). Applying the
  projection $E_e: \Pi \to \wt\Yim^e$, we obtain part (1).
\end{proof}

In the case $w=e$, the statement of this theorem is trivial. In the
case $w=s_j$, because we know the $q$-character of
$L(\Psib_{s_i(\omega_i),a})$ from Theorem \ref{fex}, we obtain the
statement of Corollary \ref{exttq2} which in turn implies Conjectures
\ref{exttq} and \ref{exttq1} when $w$ is a simple reflection (note
that for $w=e$ we have proved these conjectures in Theorems
\ref{thmFH} and \ref{corTQ}). Thus, we obtain the following.

\begin{thm}
   Conjectures \ref{exttq} and \ref{exttq1} hold if $w=s_j, j \in
   I$.
\end{thm}

%\begin{rem}
%Recall the map $E_e: \Pi \to \wt\Yim^e$. Applying it to the image
%  of the diagonal embedding $\Yim \hookrightarrow \Pi$, we obtain he
%  natural inclusion $\Yim \hookrightarrow \wt\Yim^w$. Hence we also
%  obtain a relation expressing the image of $\chi_q(V)$ in
%  $\wt{\mathcal{Y}}^e$ by replacing 
%  every variable $Y_{i,a}$ by
%$$[w(\omega_i)] \frac{[\Psib_{w(\omega_i),aq_i^{-1}}]E_e(Q_{w(\omega_i),aq_i^{-%1}})}{[\Psib_{w(\omega_i),aq_i}]E_e(Q_{w(\omega_i),aq_i})}.$$
%\end{rem}

In light of this discussion, it is natural to conjecture the following
generalization of Theorem \ref{fex} and formula \eqref{sioi} for
$\Theta_i(Y_{i,a})$.

\begin{conj}    \label{wconj}
  For any $w \in W$, we have the following identity in $\wt\Yim^e$:
  \begin{equation}    \label{qcharw}
    E_e(\Theta_w(Y_{i,a})) = [w(\omega_i)]
\frac{\chi_q(L(\Psib_{w(\omega_i),aq_i^{-1}}))}{\chi_q(L(\Psib_{w(\omega_i),aq_i}))}
  \end{equation}
\end{conj}

Theorem \ref{mfh} and Conjecture \ref{wconj} imply Conjectures
\ref{exttq} and \ref{exttq1} for an arbitrary $w \in W$.
  
At the moment, we can only prove Conjecture \ref{wconj} for $w=e$ and
$w=s_j, j \in I$ (as explained above) and for an arbitrary $w \in W$
if $\g$ has rank $2$ (see Section \ref{r2s}). However, in the next
subsection we will show that this conjecture pins down the
$q$-character of $L(\Psib_{w(\omega_i),aq_i^{-1}})$ up to
multiplication by a non-zero factor from a completion of ${\mc E}
\subset {\mc E}_\ell$ (i.e. a possibly infinite linear combination of
constant $\ell$-weights, see Section \ref{catO}). This in turn implies
Conjectures \ref{exttq} and \ref{exttq1} (see Corollary \ref{twoconj}
below).

%  By Theorem \ref{fex}, formula \eqref{defti} expresses the ratio of
%$q$-characters of modules appearing in Conjecture \ref{exttq} in the
%case $w=s_i$. The invariance of $q$-characters under $\Theta_i$ (see
%Theorem \ref{mfh}) is closely related to the statement of this
%conjecture for simple reflexions. We will make this more precise in
%Example \ref{mainm}.

\subsection{Computation of $\Theta_w(Y_{i,a})$}    \label{comptheta}

We will say that a family of $\ell$-weights $\Psib_a$ parametrized by
$a\in\mathbb{C}^\times$ is $\mathbb{C}^\times$-equivariant if $\Psib_a
= \tau_a(\Psib_1)$ for any $a\in\mathbb{C}^\times$. We have an
analogous definition of $\mathbb{C}^\times$-equivariant families in
$\Pi$. We also have a notion of $\mathbb{C}^\times$-invariant element,
that is fixed by $\tau_a$, $a\in\mathbb{C}^\times$.

Note that by Theorem \ref{mfh}, $\Theta_w(Y_{i,a})$ only
  depends on $a \in \C^\times$ and $w(\omega_i) \in {\mc
    P}_{\omega_i}$. Recall from Lemma \ref{wn},(3) that the subsets ${\mc
    P}_{\omega_i}, i \in I$, of $P$ are mutually disjoint.

\begin{thm}\label{exwo} Let $i\in I$ and $w\in W$.

(1) There is a family of 
$\ell$-weights $\Psib_{w(\omega_i),a}$ and a family of invertible elements $Q_{w(\omega_i),a}\in \Pi$, both $\mathbb{C}^\times$-equivariant and unique, such that
\begin{equation}\label{yq}\Theta_w(Y_{i,a}) =  [w(\omega_i)] \frac{\Psib_{w(\omega_i),aq_i^{-1}}Q_{w(\omega_i),aq_i^{-1}}}{\Psib_{w(\omega_i),aq_i}Q_{w(\omega_i),aq_i}},\end{equation}
with $\Lambda(Q_{w(\omega_i),a}) = 1 = \varpi(Q_{w(\omega_i),a})$ and $\Psib_{w(\omega_i),a}(0) = 1$.

(2) The $\ell$-weight $\Psib_{w(\omega_i),a}$ defined in part (1)
coincides with the $\ell$-weight defined $\Psib_{w(\omega_i),a}$
explicitly in Definition \ref{subs}.
\end{thm}

%\begin{rem} Here we use the same notation $\Psib_{w(\omega_i),a}$ as
%  we used for the $\ell$-weight defined explicitly in Section
%  \ref{esys}. In fact, we will show below that the two definitions give
%  rise to the same $\ell$-weights. Until we have established the identity,
%  this notation refer to the $\ell$-weight defined in Proposition
%  \ref{exwo}.
%\end{rem}

%\begin{rem} It follows from the characterization above that 
%$$\Lambda(\Theta_w(Y_{i,a})) = [w(\omega_i)] \frac{\Psib_{w(\omega_i),aq_i^{-1}}}{\Psib_{w(\omega_i),aq_i}}.$$
%\end{rem}

\begin{example}  For example, we have 
\begin{equation}\label{thisa}Q_{\omega_i,a} = 1, \qquad
  Q_{s_i(\omega_i),a} = \Sigma_{i,aq_i^{-2}} (1 -
  [-\alpha_i]).\end{equation}
\end{example}

\medskip

\noindent {\em Proof of Theorem \ref{exwo}.}
To prove uniqueness, consider $\Psib_{w(\omega_i),a}$ and $Q_{w(\omega_i),a}$ as in the proposition. Then 
\begin{equation}\Lambda(\Theta_w(Y_{i,a})) =  [w(\omega_i)] \frac{\Psib_{w(\omega_i),aq_i^{-1}}}{\Psib_{w(\omega_i),aq_i}}.\end{equation}
If $\Psib_{w(\omega_i),a}'$ and $Q_{w(\omega_i),a}'$ give rise
to another solution, then we setting $\Psib :=
\Psib_{w(\omega_i),1}'\Psib_{w(\omega_i),1}^{-1}$, we obtain that
$\Psib(z) = \Psib(zq_i^2)$. This implies that $\Psib$ is a constant
$\ell$-weight, and so by the conditions of the theorem, $\Psib = \Psib(0) = 1$.
Hence $\Psib_{w(\omega_i),a}$ is unique. 
%Then a solution of (\ref{yq}) with $\Lambda(Q_{w(\omega_i),a}) = 1$ can be renormalized so that $\varpi(Q_{w(\omega_i),a}) = 1$. 
Now $Q_{w(\omega_i),a}/Q_{w(\omega_i),a}'$ is a constant $H$ that does not depend on $a$. 
Hence we have $H = \varpi(H) = 1$. Thus, uniqueness is proved.

Now let $i\in I$, $a\in\mathbb{C}^\times$ and $w\in W$. Recall that $\Theta_w(Y_{i,a})$ depends only on $a$ and $w(\omega_i)$.
The existence of $\Psib_{w(\omega_i),a}$ and $Q_{w(\omega_i),a}$ is proved by induction on the length $l(w)$ of $w$. It is clear if $w = e$. 
Then consider $w s_j$ so that $l(w s_j) = l(w) + 1$. For $j \neq i$, the result is clear for $w s_j$ as $\Theta_{ws_j}(Y_{i,a}) = \Theta_w(Y_{i,a})$. 
For $w s_i$, we have 
$$\Theta_{ws_i}(Y_{i,a}) = \Theta_w (Y_{i,a}A_{i,aq_i^{-1}}^{-1})\times \frac{\Theta_w(\Sigma_{i,aq_i^{-3}})}{\Theta_w(\Sigma_{i,aq_i^{-1}})}.$$
For the second factor, we have
$$\frac{\Theta_w(\Sigma_{i,aq_i^{-3}})}{\Theta_w(\Sigma_{i,aq_i^{-1}})} = \frac{\Lambda(\Theta_w(\Sigma_{i,aq_i^{-3}}))}{\Lambda(\Theta_w(\Sigma_{i,aq_i^{-1}}))}\frac{\Theta_w(\Sigma_{i,aq_i^{-3}})\Lambda(\Theta_w(\Sigma_{i,aq_i^{-3}}^{-1}))}{\Theta_w(\Sigma_{i,aq_i^{-1}})\Lambda(\Theta_w(\Sigma_{i,aq_i^{-1}}^{-1}))},$$
which is of the correct form as $\Lambda(\Theta_w(\Sigma_{i,aq_i^{-3}}))$ is in $\mathcal{M}$ (and so is an $\ell$-weight). 

For the first factor, 
the induction hypothesis on the length of $w$ gives that for each $j\in I$, $b\in\mathbb{C}^\times$, $\Theta_w(Y_{j,b})$ is of the form (\ref{yq}). Note that the spectral parameter shift in the expression for $j$ is $q_j$, which is not necessarily equal to $q_i$. However we explain it leads to the desired expression for $\Theta_w (Y_{i,a}A_{i,aq_i^{-1}}^{-1})$. Indeed, the factor $\Theta_w(Y_{i,aq_i^{-2}}^{-1})$ is already in the correct form. Let $j$ so that $C_{j,i}\neq 0$. 

If $C_{j,i} = -1$, the factor 
$\Theta_w(Y_{j,b})$ occurs with $b = aq_i^{-1}$. If $d_j = d_i$, it is in the correct form and if $d_j = 2 > d_i = 1$, we have
$$\Theta_w(Y_{j,b}) = [w(\omega_j)] \frac{\Psib_{w(\omega_j),bq^{-2}}Q_{w(\omega_i),bq^{-2}}}{\Psib_{w(\omega_j),bq^2}Q_{w(\omega_j),bq^2}} 
=
[w(\omega_j)] 
\frac{\Psib_{w(\omega_j),bq^{-2}}Q_{w(\omega_i),bq^{-2}}\Psib_{w(\omega_j),b}Q_{w(\omega_j),b}}{\Psib_{w(\omega_j),b}Q_{w(\omega_j),b}\Psib_{w(\omega_j),ab^2}Q_{w(\omega_j),bq^2}}.
$$
(it is analogous if $d_j = 3 > d_i = 1$). 

If $C_{j,i} = -2$, we have $d_j = 1$, $d_i = 2$ and  the factor 
$\Theta_w(Y_{j,aq^{-1}}Y_{j,aq})$ occurs. It is of the correct form
$$[w(2\omega_j)] \frac{\Psib_{w(\omega_j),aq^{-2}}Q_{w(\omega_i),aq^{-2}}\Psib_{w(\omega_j),a}Q_{w(\omega_i),a}}{\Psib_{w(\omega_j),a}Q_{w(\omega_i),a}\Psib_{w(\omega_j),aq^2}Q_{w(\omega_j),aq^2}}
= [w(2\omega_j)] \frac{\Psib_{w(\omega_j),aq_i^{-1}}Q_{w(\omega_i),aq_i^{-1}}}{\Psib_{w(\omega_j),aq_i}Q_{w(\omega_j),aq_i^{-1}}}.$$
It is analogous if $C_{j,i} = -3$. This completes the proof of part
(1) of the theorem.

To prove part (2) of the theorem, we will use the
  following generalization of (\ref{til}).

%It remains to identify the $\ell$-weight $\Psib_{w(\omega_i),a}$ with
%the $\ell$-weight defined explicitly in {\color{red} Definition
%  \ref{subs}}. To do
%that.

\begin{prop}\label{tctw} (1) For any $w\in W$,

(i) we have
  \begin{equation}\label{ctw}\Lambda \circ \Theta_w = \sigma \circ T_w
  \circ \sigma ;
\end{equation}

(ii) for any $i\in I$ and $a\in\mathbb{C}^\times$, the $\ell$-weight $\Psib_{w(\omega_i),a}$ in the statement of Theorem \ref{exwo} satisfies
  \begin{equation}\label{aideform}
    [w(\omega_i)]\frac{\Psib_{w(\omega_i),aq_i^{-1}}}{\Psib_{w(\omega_i),aq_i}}
= \sigma (Y_{w(\omega_i),a^{-1}}).\end{equation}
\end{prop}

\begin{proof}
Since $T_w=\Lambda_w\circ \Theta_w$ is an automorphism of $\mathcal{M}$, we can consider its inverse $T_w^{-1}$.
For example, we have for $i,j\in I$ and $a\in\mathbb{C}^\times$,
$$T_{s_i}^{-1}(Y_{j,a}) = Y_{j,a} A_{i,aq_i^{-1}}^{-\delta_{i,j}} = (\sigma T_{s_i}  \sigma)(Y_{j,a}).$$ 
This implies that for a general $w\in W$ we have
$$T_{w^{-1}}^{-1} = \sigma \circ T_w \circ \sigma.$$
Now consider 
$$T_{w^{-1}}(m) = m_{w^{-1}}.$$ 
As $\chi_q(L(m))$ is fixed by $\Theta_w$, $m$ is the monomial of highest weight in 
$$\chi_q(L(m)) = E_e(\Theta_w(\chi_q(L(m)))).$$ 
It follows from the first part (existence) in Theorem \ref{exwo} that for a monomial $m'$ of weight $\omega$, the monomial in $E_e(\Theta_w(m'))$ of highest weight  
has weight $w(\omega)$. Let $\omega$ be the weight of $m$. Then $m_{w^{-1}}$ is the only monomial of weight $w^{-1}(\omega)$ and 
for any other monomial $m'$ in $\chi_q(L(m))$ of weight $\omega'\neq w^{-1}(\omega_m)$,  the weight $w(\omega')$ is strictly lower than $\omega$. We obtain that 
only $\Theta_w(m_{w^{-1}})$ can contribute in $E_e(\Theta_w(\chi_q(L(m))))$ to the terms of weight $\omega$. We have established
$$\Lambda (\Theta_w (m_{w^{-1}})) = m.$$
Since this is true for any dominant monomial $m$, we have
$$\Lambda\circ \Theta_w\circ T_{w^{-1}} = \text{Id},$$
that is
$$\Lambda\circ \Theta_w = T_{w^{-1}}^{-1} = \sigma \circ T_w \circ
\sigma.$$
 This proves part (i). To prove part (ii), observe that
  $\Lambda\circ
\Theta_w(Y_{i,a})$ is equal the LHS of equation \eqref{aideform}, by
the definition of $\Psib_{w(\omega_i),a}$ given in the statement of
Theorem \ref{exwo}, and $\sigma \circ T_w \circ
\sigma(Y_{i,a})$ is equal to the RHS.
\end{proof}

Now we can conclude the proof of part (2) of Theorem \ref{exwo}.
Using equation \eqref{aideform}, we obtain an expression of
$\Psib_{w(\omega_i),1}$ in terms of $\Psib_{j,b}^{\pm 1}$ using the
factorization of $Y_{w(\omega_i),1}$ in terms of $Y_{j,c}^{\pm 1}$. This
expression coincides with the one given in Definition \ref{subs}
because we have

$[\omega_k] \frac{\Psib_{k,aq_i^{-1}}}{\Psib_{k,aq_i}} = Y_{k,a}$ if $d_i = d_k$, 

$[\omega_k] \frac{\Psib_{k,aq^{-2}}\Psib_{k,a}}{\Psib_{k,a}\Psib_{k,aq^2}} = Y_{k,a}$ if $C_{i,k} = -2$, 

$[\omega_k] \frac{\Psib_{k,a q^{-3}}\Psib_{k,aq^{-1}}\Psib_{k,aq}}{\Psib_{k,aq^{-1}}\Psib_{k,a q}\Psib_{k,a q^3}} = Y_{k,a}$ if $C_{i,k} = -3$,

$[2\omega_k] \frac{\Psib_{k,aq^{-2}}}{\Psib_{k,q^2}} = Y_{k,aq^{-1}}Y_{k,aq}$ if $C_{k,i} = -2$,

$[3\omega_k] \frac{\Psib_{k, a q^{-3}}}{\Psib_{k, aq^3}} = Y_{k,aq^{-2}}Y_{k,a}Y_{k,aq^2}$ if $C_{k,i} = -3$.
\qed

\medskip

 \noindent {\em Proof of Proposition \ref{proptprime}.} From
  formula \ref{aideform}, we obtain that
$$\Psib_{w(\omega_i),aq_i^{-1}}\Psib_{w(\omega_i),aq_i}^{-1}
= \sigma (T'_w(\Psib_{i,a^{-1}q_i^{-1}}\Psib_{i,a^{-1}q_i}^{-1}))$$
and so
$$\Psib_{w(\omega_i),aq_i^{-1}} \sigma (T'_w(\Psib_{i,a^{-1}q_i}))  =
\Psib_{w(\omega_i),aq_i} \sigma (T'_w(\Psib_{i,a^{-1}q_i^{-1}}))$$
which implies that
$$\Psib_{w(\omega_i),aq_i^{-1}} (\sigma (T'_w(\sigma(\Psib_{i,aq_i^{-1}}))))^{-1}  
= \Psib_{w(\omega_i),aq_i} (\sigma (T'_w(\sigma(\Psib_{i,aq_i}))))^{-1}.$$
Hence $\Psib_{w(\omega_i),aq_i^{2m}} (\sigma \circ
T'_w \circ \sigma(\Psib_{i,aq_i^{2m}}))^{-1}, m \in \Z$, does not depend on
  $m$, and since it is a Laurent monomial in the 
$\Psib_{j,aq^r}^{\pm 1}, j \in I, r \in \Z$, it must be equal to
$1$. Hence the result.\qed

\subsection{Conjectural $q$-character formula for
  $L(\Psib_{w(\omega_i),a})$}

Theorem \ref{exwo} implies that Conjecture \ref{wconj} is
equivalent to the following conjecture (which holds for $w=e$ by
formula \eqref{pref}; for $w=s_j, j \in I$, by Theorem
\ref{fex}; and in the rank $2$ case, as we show in Section
\ref{r2s}).

\begin{conj}    \label{upto}\label{bass2} 
  We have the $q$-character formula
  \begin{equation}    \label{qupto}
    \chi_q(L(\Psib_{w(\omega_i),a})) = \Psib_{w(\omega_i),a} \;
    E_e(Q_{w(\omega_i),a}) \; \chi(L(\Psib_{w(\omega_i),a})).
  \end{equation}
\end{conj}

\begin{rem} Instead of $E_e(Q_{w(\omega_i),a})$, let us consider
$E_y(Q_{w(\omega_i),a}) \in \wt\Yim^y$ for some $y \in W$. We expect
that this is also equal to the $q$-character of an irreducible
$U_q(\wh\bb)$-module (up to a non-zero renormalization factor from
${\mc E} \subset {\mc E}_\ell$). But this module belongs to a
``twist'' of the category $\mathcal{O}$ by $y \in W$, which is
analogous to the twist of the category $\OO$ of representations of
$\g$ by $v \in W$ -- it consists of modules that are locally finite
with the respect to the twisted Borel subalgebra $y \bb y^{-1}$. These
categories have been defined and studied by Keyu Wang \cite{W2}.
\end{rem}

We now establish the extended $TQ$-relations for the right hand sides of formula \eqref{qupto}.

\begin{thm}\label{wbax}
Let $V$ be a finite dimensional representation of $U_q(\wh{\Glie})$
and $w\in W$.

(1) The $q$-character $\chi_q(V)$, viewed as an
element of the completion $\wt\Yim^e$ of $\Yim$, is invariant under
the substitution
$$
Y_{i,a} \mapsto [w(\omega_i)]
\frac{[\Psib_{w(\omega_i),aq_i^{-1}}]E_e(Q_{w(\omega_i),aq_i^{-1}})}{[\Psib_{w(\omega_i),aq_i}]E_e(Q_{w(\omega_i),aq_i})},
\qquad \forall i \in I, a \in
\C^\times.
$$

(2) The $q$-character $\chi_q(V)$, viewed as an element of $\Pi$ under
the diagonal embedding $\Yim \hookrightarrow \Pi$, is invariant under
the substitution
\begin{equation}    \label{subw}
Y_{i,a} \mapsto [w(\omega_i)] \frac{[\Psib_{w(\omega_i),aq_i^{-1}}]Q_{w(\omega_i),aq_i^{-1}}}{[\Psib_{w(\omega_i),aq_i}]Q_{w(\omega_i),aq_i}}, \qquad \forall i \in I, a \in
\C^\times.
\end{equation}
\end{thm}

\begin{proof} We follow the same argument as in the proof of Theorem
  \ref{wbax1}: Theorem \ref{mfh} implies that the $q$-character
  $\chi_q(V)$, viewed as an element of $\Pi$ under the diagonal
  embedding $\Yim \hookrightarrow \Pi$, is invariant under all
  operators $\Theta_w, w \in W$. But the right hand side of formula
  \eqref{subw} is equal to $\Theta_w(Y_{i,a})$. This proves part
  (2). Applying the projection $E_e: \Pi \to \wt\Yim^e$, we obtain
  part (1).
\end{proof}

This theorem has the following immediate corollary.

\begin{cor}    \label{twoconj}
  Conjecture \ref{upto} implies Conjectures \ref{exttq} and
  \ref{exttq1}.
\end{cor}

%The Bethe Ansatz relations are known to be related to Baxter relations, but also 
%to $Q\wt{Q}$-relations. We establish these relations for our new variables, and so we establish extended Baxter relations 
%and extended $Q\wt{Q}$-relations.
%It was our first motivation for this work : it was expected from the oper point of view that $Q\wt{Q}$-relations
%hold in an extended form parametrized by Weyl group elements (see Remark \ref{rmff}). 

\section{Extended $QQ$-systems}    \label{QQ}

  We start by recalling in Theorem \ref{orqq} the $QQ$-system in the
  Grothendieck ring $K_0(\OO)$ of the category $\mathcal{O}$
  established in \cite{FH2}. We then formulate the extended
  $QQ$-system in $K_0(\OO)$ (Conjecture \ref{extqq}).

  Note that the analogous $QQ$-systems for Yangians were studied in
  \cite{Tsuboi,Bazhanov} for type A, in \cite{ffk,EV} for type D, and
  in \cite{MV,esv} for types ADE.

  \subsection{The $QQ$-system}    \label{QQsys}

  The following theorem was proved in \cite{FH2}.
  
  \begin{thm}\label{orqq}
    For all $i\in I$ and $a\in\mathbb{C}^\times$, we have the
    following system of relations in $K_0(\OO)$ called the $QQ$-system:
\begin{equation}    \label{origQQ}
\wt{{\mb Q}}_{i,aq_i} {\mb Q}_{i,aq_i^{-1}} - [-\alpha_i]  \wt{{\mb Q}}_{i,aq_i^{-1}} {\mb Q}_{i,aq_i}$$
$$=  \prod_{j,C_{i,j} = -1} {\mb Q}_{j,a}\prod_{j,C_{i,j} = -2} {\mb Q}_{j,aq} {\mb Q}_{j,aq^{-1}}
\prod_{j,C_{i,j} = -3} {\mb Q}_{j,aq^2} {\mb Q}_{j,a} {\mb Q}_{j,aq^{-2}},
\end{equation}
where 
$${\mb Q}_{i,a} = [L(\Psib_{i,a})]/\chi(L(\Psib_{i,a})), \quad
\wt{{\mb Q}}_{i,a} =
[L(\wt{\Psib}_{i,aq_i^{-2}})]\chi_i/(\chi(L(\wt{\Psib}_{i,aq_i^{-2}}))),
\quad \chi_i = \sum_{n \geq 0} [-n\alpha_i]$$
and $\wt{\Psib}_{i,aq_i^{-2}}$ is defined by formula
\eqref{psfh2}.
\end{thm}

\begin{rem} Here we use the normalization we used in \cite{FHR}, which
  is slightly different from the normalization we used in
  \cite{FH2}. The relation between the two is explained in
  \cite[Section 5.8]{FHR}. The system \eqref{origQQ} was originally
  introduced in \cite{MRV1,MRV2} in the context of affine opers.
\end{rem}

\subsection{Generalization}

 We are going to formulate a generalization of the system
  \eqref{origQQ}, called the extended $QQ$-system, which is a system
  of relations on the variables ${\mb Q}_{w(\omega_i),a}, w \in W, i
  \in I, a \in \C^\times$. It includes equations \eqref{origQQ} in
  which we set ${\mb Q}_{i,a} = {\mb Q}_{\omega_i,a}$ and $\wt{\mb
    Q}_{i,a} = {\mb Q}_{s_i(\omega_i),a}$.

We will need the following result.

\begin{lemma}    \label{family}
There is a unique family
$$\chi_{w(\omega_i)} \in \ZZ[(1 - [-\alpha])^{-1}]_{\alpha \in \Delta}\text{ for }w\in W, i\in I$$ 
such that $\chi_{\omega_i} = 1$ and 
\begin{equation}\label{indchii}\chi_{ws_i(\omega_i)}\chi_{w(\omega_i)}
  = (1 - [- w(\alpha_i)])^{-1} \prod_{j\neq
    i}\chi_{w(\omega_j)}^{-C_{i,j}}\text{ if $l(w s_i) = l(w) +
    1$.}\end{equation}
\end{lemma}

Recall the topological completion
  $(\mathbb{Z}[y_i^{\pm 1}]_{i\in I})^v, v \in W$, of $\mathbb{Z}[y_i^{\pm
    1}]_{i\in I}$ defined in Section \ref{add}. We have a natural
  embedding
\begin{equation}    \label{emby}
{\mb e}_v: \ZZ[(1 - [-\alpha])^{-1}]_{\alpha \in \Delta}
\hookrightarrow (\mathbb{Z}[y_i^{\pm 1}]_{i\in I})^v
\end{equation}
obtained by expanding each $(1 - [-\alpha])^{-1}, \alpha \in \Delta$,
in positive powers of $[-\alpha]$ if $\alpha \in v(\Delta_+)$ and in
positive powers of $[\alpha]$ if $\alpha \in v(\Delta_+)$.

In what follows, we will consider $\chi_{w(\omega_i)}$ as an element
the ring
\begin{equation}    \label{piagain}
  \ol\Pi' = \bigoplus_{v \in W} \ol{\mathcal{Y}}'{}^v
\end{equation}
  introduced in Section \ref{add}, whose component
in $\ol{\mathcal{Y}}'{}^v$ is ${\mb e}_v(\chi_{w(\omega_i)})$ (note
that according to the definitions given in Section \ref{add}, we have
natural inclusion
$(\mathbb{Z}[y_i^{\pm 1}]_{i\in I})^v \subset
\ol{\mathcal{Y}}'{}^v$).

\begin{rem} With this interpretation of $\chi_{w(\omega_i)}$, the
  relations (\ref{indchii}) coincide with
   the image of the extended $QQ$-system \eqref{origQQ} under the map
   $\varpi$ (i.e. we forget the spectral parameters in
   (\ref{indchii})).
\end{rem}

\noindent {\em Proof of Lemma \ref{family}.}
By induction on the length of $w$, we prove that there
  exists a unique solution of this system that {\em a priori} depends
  on both $i \in I$ and $w \in W$. Moreover, it follows by induction
  that the highest weight term of
  $$
{\mb e}_e(\chi_{w(\omega_i)}) \in (\mathbb{Z}[y_i^{\pm 1}]_{i\in
  I})^e = \ZZ[[y_i^{-1}]]_{i\in I}
$$
is equal to $1$. It remains to show that this solution only depends on
  $w(\omega_i)$.

To do that, recall the operators $\Theta_w'$ defined in Proposition
\ref{extact1}.  Since $\Sigma_{i,a} =
\wt{\Psib}_{i,a}^{-1}\Theta_i'(\Psib_{i,aq_i^2})$ and $A_{i,a} =
      [\alpha_i]\wt{\Psib}_{i,a}\Psib_{i,a}\wt{\Psib}_{i,aq_i^{-2}}^{-1}
      \Psib_{i,aq_i^2}^{-1}$,
      applying $\varpi \circ \Theta_w'$ to equation (\ref{siae}), we
      obtain that $\chi_{w(\omega_j)} =
      \varpi(\Theta_w'(\Psib_{i,a})), w \in W$, is a solution of the system
      \eqref{indchii}.

Now, by Proposition \ref{actG} and Lemma \ref{wn},
$\varpi(\Theta_w'(\Psib_{i,a}))$ depends on $w(\omega_i)$ up to a
factor in $\pm \tb^\times$. But the leading term of 
${\mb e}_e(\varpi(\Theta_w'(\Psib_{i,a}))) =
  \varpi(E_e(\Theta_w'(\Psib_{i,a})))$
must be equal to $1$ because it is a solution of our system, as we
have shown above. Therefore,
$\varpi(\Theta_w'(\Psib_{i,a}))$ depends only on $w(\omega_i)$.\qed

\medskip

\begin{example} For $i \in I$ and $w=s_i$, we have
    $$\chi_{s_i(\omega_i)} = (1 - [-\alpha_i])^{-1}.$$
    Therefore
    $$
    {\mb e}_e(\chi_{s_i(\omega_i)}) = \sum_{n\geq 0} [-n\alpha_i]
    $$
    which is the series $\chi_i$ from Theorem \ref{orqq}.
\end{example}

For $\alpha\in \Delta$, we set $[\alpha]_+ = 1$ if $\alpha\in\Delta_-$ and $[\alpha]_+ = - [\alpha]$ if $\alpha\in \Delta_-$.

Recall the ring $\overline{\Pi}'$ defined in Section \ref{extact}. We introduce
the following elements of $\overline{\Pi}'$ for $w\in W$, $i\in I$, and
$a\in\mathbb{C}^\times$:
\begin{equation}    \label{Qw}
\mathcal{Q}_{w(\omega_i),a}' = \Psib_{w(\omega_i),a}Q_{w(\omega_i),a},
\qquad \mathcal{Q}_{w(\omega_i),a} =
\chi_{w(\omega_i)}\mathcal{Q}_{w(\omega_i),a}'.
\end{equation}

According to Conjecture \ref{bass2}, we have
  \begin{equation}    \label{Qw1}
    E_e(\mathcal{Q}_{w(\omega_i),a}) = \chi_q(L(\Psib_{w(\omega_i),a})) \cdot
    {\mb e}_e(\chi_{w(\omega_i)}/\chi(L(\Psib_{w(\omega_i),a})).
  \end{equation}

\begin{thm}\label{QQrel} For all $i\in I$, $w\in W$, $a\in\mathbb{C}^\times$, we have the extended $QQ$-system in $\overline{\Pi}'$
$$\mathcal{Q}_{ws_i(\omega_i),aq_i} \mathcal{Q}_{w(\omega_i),aq_i^{-1}} - 
[-w(\alpha_i)]  \mathcal{Q}_{ws_i(\omega_i),aq_i^{-1}}\mathcal{Q}_{w(\omega_i),aq_i}$$
$$= [-w(\alpha_i)]_+
   \prod_{j,C_{i,j} = -1}\mathcal{Q}_{w(\omega_j),a}\prod_{j,C_{i,j} = -2}\mathcal{Q}_{w(\omega_j),aq}\mathcal{Q}_{w(\omega_j),aq^{-1}}
\prod_{j,C_{i,j} = -3}\mathcal{Q}_{w(\omega_j),aq^2}\mathcal{Q}_{w(\omega_j),a}\mathcal{Q}_{w(\omega_j),aq^{-2}},$$
where $[-w(\alpha_i)]_+ = 1$ if $w(\alpha_i)\in\Delta_+$ and $[-w(\alpha_i)]_+ = - [-w(\alpha_i)]$ if $w(\alpha_i)\in \Delta_-$.
\end{thm}

\begin{rem} For each $w'\in W$, by applying $E_{w'}$ to the extended
  $QQ$-system in Theorem \ref{QQrel},
we obtain a relation between the expansions
$E_{w'}(\mathcal{Q}_{w(\omega_j),a})$ in the ring
$\overline{\mathcal{Y}}'{}^{w'}$ introduced in Section \ref{extact}.
Our proof below is given in the 
ring $\overline{\Pi}'$ which is the direct sum of
$\overline{\mathcal{Y}}'{}^{w'}$ over all $w' \in W$. Hence we also
obtain the proof of the corresponding relations in
$\overline{\mathcal{Y}}'{}^{w'}$for each $w'$.
\end{rem}

\noindent {\em Proof of Theorem \ref{QQrel}.}
First let us consider
\begin{equation}\label{defsi}\Sigma_{i,w,a} = \Theta_w(\Sigma_{i,a}) Q_{ws_i(\omega) ,aq_i^2}^{-1}.\end{equation}
We can express $\Sigma_{i,w,a}$ in terms of the $Q_{w(\omega_j),b}$ in the following way. 

\begin{lem}\label{lsw} For any $i\in I$, $w\in W$, $a\in\mathbb{C}^\times$, we have
$$\frac{\Lambda(\Sigma_{i,w,a}^{-1})\Sigma_{i,w,a}}{\varpi(\Lambda(\Sigma_{i,w,a}^{-1})\Sigma_{i,w,a})} =$$
$$Q_{w(\omega_i),a}\prod_{j,C_{i,j} = -1}Q_{w(\omega_j),aq_i}^{-1}\prod_{j,C_{i,j} = -2}Q_{w(\omega_j),aq^{2}}^{-1}Q_{w(\omega_j),a}^{-1}
\prod_{j,C_{i,j} = -3}Q_{w(\omega_j),aq^{3}}^{-1}Q_{w(\omega_j),aq}^{-1}Q_{w(\omega_j),aq^{-1}}^{-1}.$$
\end{lem}

\begin{proof} We have
$$\Theta_{ws_i}(Y_{i,a}) =  \Theta_{w}(Y_{i,a}A_{i,aq_i^{-1}}^{-1} )\frac{\Theta_{w}(\Sigma_{i,aq_i^{-3}})}{\Theta_{w}(\Sigma_{i,aq_i^{-1}})} = \Lambda(\Theta_{ws_i}(Y_{i,a})) \frac{Q_{ws_i(\omega_i),aq_i^{-1}}}{Q_{ws_i(\omega_i),aq_i}},$$
$$\Theta_{w}(Y_{i,a}A_{i,aq_i^{-1}}^{-1} ) = \Lambda(\Theta_{ws_i}(Y_{i,a})) \frac{\Sigma_{i,w,aq_i^{-1}}}{\Sigma_{i,w,aq_i^{-3}}},$$
\begin{equation}\label{twa}\Theta_w(A_{i,a}^{-1}) = \Lambda(\Theta_{ws_i}(Y_{i,a})\Theta_w(Y_{i,a}^{-1}))  \frac{\Sigma_{i,w,a}Q_{w(\omega_i),a q_i^2}}{\Sigma_{i,w,aq_i^{-2}}Q_{w(\omega_i),a}}.\end{equation}
But from (\ref{yq}) and the definition of $A_{i,a}$, we obtain 
\begin{multline*}
\Theta_w(A_{i,a})\Lambda(\Theta_w(A_{i,a}^{-1})) =
\\ \frac{Q_{w(\omega_i),aq_i^{-2}}}{Q_{w(\omega_i),aq_i^2}}\prod_{j,C_{j,i}
  = -1}\frac{Q_{w(\omega_j),aq_j}}{Q_{w(\omega_j),aq_j^{-1}}}
\prod_{j,C_{j,i} = -2}\frac{Q_{w(\omega_j),aq^2}}{Q_{w(\omega_j),aq^{-2}}}
\prod_{j,C_{j,i} =
  -3}\frac{Q_{w(\omega_j),aq^3}}{Q_{w(\omega_j),aq^{-3}}},
\end{multline*}
which gives the formula in the lemma.
\end{proof}

Now we can finish the proof of Theorem \ref{QQrel}. By (\ref{siae}), we have
$$\Sigma_{i,a} = 1 + A_{i,a}^{-1}\Sigma_{i,aq_i^{-2}}\text{ and so }
\Theta_w(\Sigma_{i,a}) = 1 + \Theta_w(A_{i,a}^{-1}) \Theta_{w}(\Sigma_{i,aq_i^{-2}}),$$
$$Q_{ws_i(\omega_i),aq_i^2}\Sigma_{i,w,a} = 1 +  \Lambda(\Theta_{ws_i}(Y_{i,a})\Theta_w(Y_{i,a}^{-1}))  \frac{\Sigma_{i,w,a}Q_{w(\omega_i),a q_i^2}}{\Sigma_{i,w,aq_i^{-2}}Q_{w(\omega_i),a}} Q_{ws_i(\omega_i),a}\Sigma_{i,w,aq_i^{-2}}.$$
Then we obtain
$$Q_{ws_i(\omega_i),aq_i^2} Q_{w(\omega_i),a} = \frac{Q_{w(\omega_i),a}}{\Sigma_{i,w,a}} + \Lambda(\Theta_{ws_i}(Y_{i,a})\Theta_w(Y_{i,a}^{-1}))  Q_{ws_i(\omega_i),a}Q_{i,w,aq_i^2}.$$
After a shift of spectral parameters by $q_i^{-1}$, using Lemma \ref{lsw}, we obtain that
$$\Lambda(\Sigma_{i,w,aq_i^{-1}}) Q_{ws_i(\omega_i),aq_i} Q_{w(\omega_i),aq_i^{-1}} -  \Lambda(\Sigma_{i,w,aq_i^{-3}})\Lambda(\Theta_w(A_{i,aq_i^{-1}}^{-1})) Q_{ws_i(\omega_i),aq_i^{-1}}Q_{w(\omega_i),aq_i}$$
$$= \varpi(\Lambda(\Sigma_{i,w,a})\Sigma_{i,w,a}^{-1}) \prod_{j,C_{i,j} = -1}Q_{w(\omega_j),a}\prod_{j,C_{i,j} =
  -2}Q_{w(\omega_j),aq}Q_{w(\omega_j),aq^{-1}}  \times$$
$$\times \prod_{j,C_{i,j} =
  -3}Q_{w(\omega_j),aq^2}Q_{w(\omega_j),a}Q_{w(\omega_j),aq^{-2}}.$$

We have $\Lambda(\Sigma_{i,w,a}) = \Lambda(\Theta_w(\Sigma_{i,a}))$ as $\Lambda(Q_{ws_i(\omega_i),a}) = 1$. 
%Moreover
%$$\Theta_{ws_i}(Y_{i,a}) = \Theta_w(Y_{i,a}A_{i,aq_i^{-1}}) \frac{\Theta_w(\Sigma_{i,aq_i^{-3}})}{\Theta_w(\Sigma_{i,aq_i^{-1}})}.$$
Formula (\ref{twa}) implies that
\begin{equation}\label{ltwa1}\Lambda(\Theta_w(A_{i,aq_i^{-1}}^{-1})) = [-w(\alpha_i)]\frac{\Psib_{ws_i(\omega_i),aq_i^{-1}}\Psib_{w(\omega_i),aq_i}\Lambda(\Sigma_{i,w,aq_i^{-1}})}
{\Psib_{ws_i(\omega_i),aq_i}\Psib_{w(\omega_i),aq_i^{-1}}\Lambda(\Sigma_{i,w,aq_i^{-3}})}.\end{equation}
Besides, from the explicit expression of $A_{i,a}$, the same computation as above gives 
\begin{equation}\label{ltwa}\Lambda(\Theta_w(A_{i,a})) = [w(\alpha_i)]\frac{\wt{\Psib}_{w(\omega_i),aq_i^{-2}}^{-1}\Psib_{w(\omega_i),a}}{\wt{\Psib}_{w(\omega_i),a}^{-1}\Psib_{w(\omega_i),aq_i^{2}}}\end{equation}
where
$$\wt{\Psib}_{w(\omega_i),aq_i^{-1}} = \Psib_{w(\omega_i),aq_i^{-1}}^{-1}$$
$$\times \prod_{j,C_{i,j} = -1}\Psib_{w(\omega_j),a}\prod_{j,C_{i,j} = -2}\Psib_{w(\omega_j),aq}\Psib_{w(\omega_j),aq^{-1}}
\prod_{j,C_{i,j} = -3}\Psib_{w(\omega_j),aq^2}\Psib_{w(\omega_j),a}\Psib_{w(\omega_j),aq^{-2}}.$$
 Note that $\wt{\Psib}_{\omega_i,a}$ is equal to $\wt{\Psib}_{i,a}$
from Example \ref{exPsi},(ii). Therefore from (\ref{ltwa1}) and
(\ref{ltwa}) we obtain that
$$\Lambda(\Sigma_{i,w,a}) = \varpi(\Lambda(\Sigma_{i,w,a}))\wt{\Psib}_{w(\omega_i),a}^{-1}\Psib_{ws_i(\omega_i),aq_i^{2}}.$$
This gives
$$\Psib_{ws_i(\omega_i),aq_i}\Psib_{w(\omega_i),aq_i^{-1}} Q_{ws_i(\omega_i),aq_i} Q_{w(\omega_i),aq_i^{-1}} 
- $$ $$
[-w(\alpha_i)]\Psib_{ws_i(\omega_i),aq_i^{-1}}\Psib_{w(\omega_i),aq_i}
Q_{ws_i(\omega_i),aq_i^{-1}}Q_{w(\omega_i),aq_i}$$
$$= \Psib_{ws_i(\omega_i),aq_i}\Psib_{w(\omega_i),aq_i^{-1}} \wt{\Psib}_{w(\omega_i),aq_i^{-1}}\Psib_{(ws_i)(\omega_i),aq_i}^{-1}
 \varpi(\Sigma_{i,w,a}^{-1}) \times$$
$$\times \prod_{j,C_{i,j} = -1}Q_{w(\omega_j),a}\prod_{j,C_{i,j} = -2}Q_{w(\omega_j),aq}Q_{w(\omega_j),aq^{-1}}
\prod_{j,C_{i,j} = -3}Q_{w(\omega_j),aq^2}Q_{w(\omega_j),a}Q_{w(\omega_j),aq^{-2}}.$$
Note that 
$$\varpi(\Sigma_{i,w,a}^{-1}) = \varpi(\Theta_w(\Sigma_{i,a}^{-1})) = w(\varpi(\Sigma_{i,a}^{-1})) = (1 - [-w(\alpha_i)]).$$
Therefore, we have
$$\mathcal{Q}_{ws_i(\omega_i),aq_i}' \mathcal{Q}_{w(\omega_)i,aq_i^{-1}}' - 
[-w(\alpha_i)]  \mathcal{Q}_{ws_i(\omega_i),aq_i^{-1}}'\mathcal{Q}_{w(\omega_i),aq_i}'$$
$$= 
 (1 - [-w(\alpha_i)])  \prod_{j,C_{i,j} =
  -1}\mathcal{Q}_{w(\omega_j),a}'\prod_{j,C_{i,j} =
  -2}\mathcal{Q}_{w(\omega_j),aq}'\mathcal{Q}_{w(\omega_j),aq^{-1}}'
\times$$ $$\times
\prod_{j,C_{i,j} = -3}\mathcal{Q}_{w(\omega_j),aq^2}'\mathcal{Q}_{w(\omega_j),a}'\mathcal{Q}_{w(\omega_j),aq^{-2}}',$$
which is the extended $QQ$-system. Theorem \ref{QQrel} is proved.\qed

\subsection{Extended $QQ$-system Conjecture}

Conjecture \ref{bass2}, injectivity of the $q$-character
homomorphism, and Theorem \ref{QQrel} (in which we apply the
projection $E_e$ to the $Q_{w(\omega_i),a}$), imply the following
conjecture, which we call the Extended $QQ$-system Conjecture. It
describes a system of relations in $K_0(\OO)$ generalizing the
$QQ$-system of Theorem \ref{QQrel}. In Theorem \ref{r2}, we will
establish it together with Conjecture \ref{bass2} for all $\g$ of rank
$2$.

In what follows we will use the notation
  \begin{equation}    \label{ee}
  \bchi_{w(\omega_i)} := {\mb e}_e(\chi_{w(\omega_i)}).
  \end{equation}

\begin{conj}\label{extqq}
  Let $w\in W$. For any $i\in I$ and $a\in\mathbb{C}^\times$, we have in $K_0(\OO)$ the relations
\begin{equation}    \label{extqq1}
{\mb Q}_{(ws_i)(\omega_i),aq_i} {\mb Q}_{w(\omega_i),aq_i^{-1}} - 
[-w(\alpha_i)]
{\mb Q}_{(ws_i)(\omega_i),aq_i^{-1}}{\mb Q}_{w(\omega_i),aq_i} = [-w(\alpha_i)]_+ \cdot
\end{equation}
$$
 \cdot \prod_{j,C_{i,j} = -1}{\mb Q}_{w(\omega_j),a}\prod_{j,C_{i,j} = -2}{\mb Q}_{w(\omega_j),aq}{\mb Q}_{w(\omega_j),aq^{-1}}
\prod_{j,C_{i,j} = -3}{\mb Q}_{w(\omega_j),aq^2}{\mb Q}_{w(\omega_j),a}{\mb Q}_{w(\omega_j),aq^{-2}},$$
where
\begin{equation}    \label{Qw2}
{\mb Q}_{w(\omega_i),a} = [L(\Psib_{w(\omega_i),a})] \cdot
\bchi_{w(\omega_i)}/\chi(L(\Psib_{w(\omega_i),a})).
\end{equation}
\end{conj}

%\section{$q$-character formula}\label{qcharall}

%\subsection{Borel algebra}\label{conjform}

%\begin{conj}\label{bass2} Let $i\in I$, $a\in\mathbb{C}^\times$ and $w\in W$.
%We have 
%$$\chi_q(L(\Psib_{w(\omega_i),a}))/\chi(L(\Psib_{w(\omega_i),a})) = [\Psib_{w(\omega_i),a}]  E_e(Q_{w(\omega_i),a}).$$
%\end{conj}

%This Conjecture \ref{bass2}, with Theorem \ref{wbax}, implies Conjecture \ref{exttq}.

%\begin{example} For $w = e$, we recover the Baxter relations in \cite{FH} obtained by the replacement of $Y_{i,a}$ by
%$$[\omega_i] \frac{[L(\Psib_{i,aq_i^{-1}})]}{[L(\Psib_{i,aq_i})]}.$$
%This case is distinguished because as $Y_{i,a}$ is equals
%$$Y_{i,a} = \chi_q([\omega_i]) \frac{\chi_q(L(\Psib_{i,aq_i^{-1}}))}{\chi_q(L(\Psib_{i,aq_i}))}.$$
%(this is not true for general $w$). Both Conjectures \ref{bass2} and \ref{exttq} are proved in this case as $Q_{\omega_i,a} = 1$ and by \cite{FH} we have
%$$\chi_q(L(\Psib_{\omega_i,a})) = [\Psib_{\omega_i,a}] \chi(L(\Psib_{\omega_i,a})).$$
%\end{example}

%\begin{example}\label{mainm} Consider a simple reflection $w = s_i$. By Theorem \ref{fex}, both Conjectures \ref{bass2} and \ref{exttq} are proved in this case from (\ref{thisa}). 
%This was the main motivation for the definition of the Weyl group symmetry of $q$-characters in \cite{FH3} as discussed above.
%\end{example}

%These two examples imply that the conjectures are true in the $\sw_2$-case. We will establish these conjectures for rank $2$ Lie algebras in Theorem \ref{r2}.

\section{Shifted quantum affine algebras}    \label{shifted}

 A drawback of the conjectural formula \eqref{qupto} for
  the $q$-character $\chi_q(L(\Psib_{w(\omega_i),a}))$ is that for a
  general Lie algebra $\g$ we
  don't have an explicit expression for the third factor on the right
  hand side of \eqref{qupto}; namely, the ordinary character
  $\chi(L(\Psib_{w(\omega_i),a}))$ (not even for $w=e$ or $s_j, \in
  I$). This is not important for the extended $TQ$-relations because
  $\chi_q(L(\Psib_{w(\omega_i),a}))$ enters in them through the ratios
  $$
  \dfrac{\chi_q(L(\Psib_{w(\omega_i),aq_i^{-1}}))}{\chi_q(L(\Psib_{w(\omega_i),aq_i}))}
  $$
  in which the factors $\chi(L(\Psib_{w(\omega_i),a}))$ cancel
  out. But the variables ${\mb Q}_{w(\omega_i),a}$ appearing in
  the $QQ$-system of Conjecture \ref{extqq} correspond to a
  particular normalization of $\chi_q(L(\Psib_{w(\omega_i),a}))$ (or
  equivalently, $[L(\Psib_{w(\omega_i),a})]$) given by formulas
  \eqref{Qw1} and \eqref{Qw2}. This suggests that perhaps the factor
  $\bchi_{w(\omega_i)}/\chi(L(\Psib_{w(\omega_i),a}))$ appearing in
  these formulas is equal to $1$. However, this is not the case, even
  for $w=e$, when $\chi_{\omega_i}=1$ but
  $\chi(L(\Psib_{\omega_i,a}))$ is known to be an infinite series.

  Nonetheless, it turns out that there are certain modifications of
  $U_q(\ghat)$ called {\em shifted quantum affine algebras} for which
  there are analogous representations $L^{\on{sh}}(\Psib_{w(\omega_i),a}))$
  satisfying (conjecturally, in general) the extended $TQ$- and
  $QQ$-systems and for which we expect that
  \begin{equation}    \label{Lsh}
  \chi(L^{\on{sh}}(\Psib_{w(\omega_i),a})) = \bchi_{w(\omega_i)}, \qquad
  \forall w \in W.
  \end{equation}
  
In this section we recall the definition of the shifted quantum affine
algebras and their categories of representations analogous to the
category $\OO$, and formulate the analogues of the extended $TQ$- and
$QQ$-systems for representations from these categories. We also
conjecture $q$-character formulas for these representations in
Conjecture \ref{cosh}. In the next section, we will establish these
conjectures for all Lie algebras of rank $2$.

\subsection{Representations of shifted quantum affine algebras}

We collect some information on the shifted quantum affine algebras of
\cite{FT} and of their representations (see \cite{Hsh} for more
details and references).

Fix an integral coweight $\mu$ of $\mathfrak{g}$. The
  following definition is due to \cite{FT}.

\begin{defi} The {\em shifted
quantum affine algebra} $U_q^\mu(\wh{\mathfrak{g}})$
  is the algebra with the same Drinfeld generators (\ref{dgen}) as the
  algebra $U_q(\wh{\mathfrak{g}})$ and the same relations except that
  the definition of the series $\phi_i^-(z)$ given in (\ref{phrel}) is
  modified as follows:
\begin{equation}\label{phrelshm}\phi_i^-(z) =
  \sum_{m\in\mathbb{Z}}\phi_{i, m}^- z^{ m} =
  \phi_{i,\alpha_i(\mu)}^-z^{\alpha_i(\mu)} \exp\left(- (q_i -
  q_i^{-1})\sum_{r > 0} h_{i,-r } z^{- r} \right)\end{equation}
(the definition of the series $\phi_i^+(z)$ stays the same). In
addition, we impose the condition that $\phi_{i,0}^+$ and
$\phi_{i,\alpha_i(\mu)}^-$ are invertible (but they are not
necessarily inverse to each other).
\end{defi}
		
Note that in particular, for $i\in I$, we have set $\phi_{i,m}^+ = 0$
for $m < 0$ and $\phi_{i,m}^- = 0$ for $m  > \alpha_i(\mu)$.

We define the category $\mathcal{O}_\mu$ of
$U_q^\mu(\hat{\mathfrak{g}})$ modules similarly to the
category $\mathcal{O}$ of $U_q(\wh\bb)$-modules in Definition
\ref{defo}. The notions of Cartan-diagonalizable 
representation and weight spaces $V_{\omega}$ are the same as before.

\begin{defi}\cite{Hsh} A $U_q^\mu(\hat{\mathfrak{g}})$-module $V$ 
is said to be in category $\mathcal{O}_\mu$ if

i) $V$ is Cartan-diagonalizable;

ii) for all $\omega\in \tb^\times$ we have 
$\dim (V_{\omega})<\infty$;

iii) there exist a finite number of elements 
$\lambda_1,\cdots,\lambda_s\in \tb^\times$ 
such that the weights of $V$ are in 
$\underset{j=1,\cdots, s}{\bigcup}D(\lambda_j)$.
\end{defi}

\begin{thm}\cite{Hsh} The simple representations $L^{\on{sh}}(\Psib)$ in $\mathcal{O}_\mu$ are parametrized by rational $\ell$-weights 
$\Psib = (\Psi_i(z))_{i\in I}$  in $\mathfrak{r}$ such that 
$$\mu = \sum_{i\in I} \on{deg}(\Psi_i(z)) \omega_i^\vee\in P^\vee.$$ 
\end{thm}

Each representation $V$ in the category $\mathcal{O}_\mu$ has a
$q$-character
$$\chi_q(V) = \sum_{\Psibs \in
  \mathfrak{r}_\mu}\text{dim}(V_{\Psibs})[\Psib]$$ where $V_{\Psibs}$ is
the $\ell$-weight space of $\ell$-weight $\Psib$ as above and
$\mathfrak{r}_\mu$ is the group of rational
  $\ell$-weights $\Psib = (\Psi_i(z))_{i\in I}$ such that
  $\on{deg}(\Psi_i(z)) = \alpha_i(\mu)$ for each $i\in I$.

\begin{thm}\cite{Hsh} The $\ell$-weights of a simple representations
  in $\mathcal{O}_\mu$ belong to $\mathfrak{r}_\mu$.
\end{thm}

As in the case of $U_q(\wh{\bb})$, the $q$-character map extends to an injective group homomorphism
\begin{equation}    \label{chiqmu}
  \chi_q : K_0(\mathcal{O}_\mu)\rightarrow \mathcal{E}_\ell.
\end{equation}

Using the fusion product of simple representations constructed from
the deformed Drinfeld coproduct, the following result was established
in \cite{Hsh}.  Note that this fusion procedure is known to work for
simple representations but not for arbitrary representations in
$\mathcal{O}_\mu$.

  \begin{thm}    \label{shmult}
    The direct sum ${\mb K} := \bigoplus_{\mu \in P^\vee}
    K_0(\mathcal{O}_\mu)$ has a ring structure such that the product
    of elements from $K_0(\mathcal{O}_\mu)$ and $K_0(\mathcal{O}_\nu)$
    belongs to $K_0(\mathcal{O}_{\mu+\nu})$ and the maps
    $\chi_q$ given by \eqref{chiqmu} combine into an injective ring
    homomorphism from ${\mb K}$ to $\mathcal{E}_\ell$.
    Moreover, the structure constants on simple representations
    are positive integers.
  \end{thm}
	
  \begin{rem}
    The ring ${\mb K}$ will be studied from the point of view of cluster
    algebras in \cite{GHL}.
    \end{rem}

\subsection{Langlands dual Lie algebra}
We have already seen the occurrence of the Langlands dual Lie algebra
${}^L\mathfrak{g}$ in the context of the $\ell$-weights
$\Psib_{w(\omega_i),a}$ (see Remark \ref{orbit}). In the context of representations of shifted quantum
  affine algebras recalled in the preceding subsection, we also assign
  the coweight $\omega_k^\vee$ to $\Psib_{k,a}$ that suggests a
  connection to ${}^L\mathfrak{g}$.

Now for $i\in I$, define $\phi_i : P_\mathbb{Q}\rightarrow
P^\vee_{\mathbb{Q}}$ as a homomorphism of $W$-modules given by the formula
$$\phi_i(\omega_k) = \frac{d_k}{d_i}\omega_k^\vee, \qquad k \in I.$$
Then we have 
$$\phi_i(\alpha_k) = \sum_{j\in I} C_{j,k} \frac{d_j}{d_i} \omega_j^\vee = \frac{d_k}{d_i}\sum_{j\in I} C_{k,j}\omega_j^\vee = \frac{d_k}{d_i} \alpha_k^\vee.$$
Indeed, 
$$\phi_i(s_j(\omega_k)) = \phi_i(\omega_k -\delta_{j,k}\alpha_k) = \frac{d_k}{d_i} (\omega_k^\vee - \delta_{j,k}\alpha_k^\vee) 
= \frac{d_k}{d_i}(s_j(\omega_k^\vee)) = s_j(\phi_i(\omega_k)).$$
In particular for $w\in W$, we have
$$\phi_i(w(\omega_i)) = w(\omega_i^\vee)$$
and so 
$$\phi_i( W \cdot \omega_i) = W \cdot \omega_i^\vee \subset P^\vee.$$
We obtain the following. 

\begin{prop}\label{lbp} Let $i\in I$, $w\in W$ and $a\in\mathbb{C}^\times$. 

(1) The coweight associated to $\Psib_{w(\omega_i),a}$ is 
$$w(\omega_i^\vee) = \phi_i(w(\omega_i)),$$ 
that is the image of the weight of $Y_{w(\omega_i)}(z)$ by $\phi_i$. 

(2) The coweight $\mu$ associated to $L^{\on{sh}}(\Psib_{w(\omega_i),a})$ is 
$$\mu = w(\omega_i^\vee).$$
Thus, $L^{\on{sh}}(\Psib_{w(\omega_i),a})$ is a simple
  object of the category ${\mathcal O}^{w(\omega_i^\vee)}$.
\end{prop}

Part (2) follows from part (1) and the definition of coweights of
simple representations of shifted quantum affine algebras.

\subsection{Conjectural $q$-character formula} We now formulate a more
explicit version of Conjecture \ref{bass2} for shifted quantum affine
algebras.

%\begin{rem} The coweight associated to $\Psib_{w(\omega_i),a}$ will be computed in Proposition \ref{lbp}.
%\end{rem}

\begin{conj}\label{cosh} The $q$-character of the simple
  representation $L^{\on{sh}}(\Psib_{w(\omega_i),a})$ is
$$\chi_q(L^{\on{sh}}(\Psib_{w(\omega_i),a})) = [\Psib_{w(\omega_i),a}]
E_e(Q_{w(\omega_i),a}) \bchi_{w(\omega_i)}.$$
  In particular, its character is equal to $\bchi_{w(\omega_i)}$:
$$
  \chi(L^{\on{sh}}(\Psib_{w(\omega_i),a}))
  = \bchi_{w(\omega_i)}.
  $$
\end{conj}

\begin{example} (i) The conjecture is true for $w = e$ because it is
  known that  $L^{\on{sh}}(\Psib_{i,a})$ is a one-dimensional
  representation by \cite[Example 4.13]{Hsh}.

(ii) The conjecture is true for a simple reflexion $w = s_i$ because
  by \cite[Example 5.2]{Hsh}, we have
$$ \chi_q(L(\wt{\Psib}_{i,a})) = \wt{\Psib}_{i,a}
    E_e(\Sigma_{i,a}).$$ 
\end{example}

\begin{rem}  Note that Conjecture \ref{cosh} implies positivity of $\bchi_{w(\omega_i)}$ for any $w\in W$, $i\in I$, which is
a purely combinatorial statement.
\end{rem}

 Theorem \ref{QQrel} and Conjecture \ref{cosh} imply the
  following analogue of Conjecture \ref{extqq} for shifted quantum
  affine algebras (here we use the multiplicative
    structure of Theorem \ref{shmult}).
  
  \begin{conj}    \label{shQQ}
The extended $QQ$-system \eqref{extqq1} is satisfied if we set
    \begin{equation}    \label{Qw3}
{\mb Q}_{w(\omega_i),a} = [L^{\on{sh}}(\Psib_{w(\omega_i),a})].
\end{equation}
    \end{conj}

The proof of Conjecture \ref{cosh} (which implies Conjecture
\ref{shQQ}) for rank $2$ Lie algebras will be given in the next
section.

\section{Rank $2$ Lie algebras}\label{r2s}

 We prove Conjecture \ref{bass2} (hence also \ref{exttq},
  \ref{exttq1}, and \ref{extqq}) and Conjecture \ref{cosh} (hence also
  \ref{shQQ}) for all rank $2$ Lie algebras.

\begin{thm}\label{r2} Conjectures \ref{bass2} and \ref{cosh} hold for
  rank $2$ Lie algebras.
\end{thm}

This is clear for type $A_1\times A_1$ from the case of $\sw_2$
discussed above.

For $i\in I$, $a\in\mathbb{C}^\times$ and $\alpha\geq 0$, define the
element of $\mathcal{M}$, 
$$  
  V_{i,a}^{(\alpha)} := (A_{i,a}A_{i,aq_i^{-2}}\cdots
  A_{i,aq_i^{2-2\alpha}})^{-1}, \quad \al>0, \qquad V_{i,a}^{(0)} := 1.
  $$
We will also use the notation $a_i = [\alpha_i]$. 

\medskip

{\bf Type $A_2$} 

\medskip

 It suffices to consider the case $i = 1$ from the symmetry of the Dynkin quiver. We have :  
$$\chi_{\omega_1} = 1\text{ , }\chi_{\omega_2 - \omega_1} = \frac{1}{1 - a_1^{-1}} \text{ , }
 \chi_{-\omega_2}  = \frac{1}{(1- a_1^{-1}a_2^{-1})(1-a_2^{-1})},$$
$$\Psib_{\omega_1,a} = \Psib_{1,a}\text{ , }
\Psib_{\omega_2 - \omega_1,a} = \Psib_{1,aq^{-2}}^{-1}\Psib_{2,a}\text{ , }
\Psib_{-\omega_2,a} = \Psib_{2,aq^{-3}}^{-1}.$$
Now the cases $w = e, s_2, s_1, s_1s_2$ follow from the discussion 
above. It suffices to treat the case $w = s_2s_1$. We have 
by \cite[Section 4.6]{FH3}
$$E_e(Q_{w(\omega_1),a}) = \sum_{0\leq \beta \leq \alpha} V_{2,aq^{-3}}^{(\alpha)}V_{1,aq^{-2}}^{(\beta)},$$
$$\bchi_{w(\omega_1)} = \varpi(E_e(Q_{w(\omega_1),a})) = {\mb e}_e\left(\frac{1}{(1 - a_2^{-1})(1 - a_1^{-1}a_2^{-1})}\right).$$
This can be compared with \cite[Example 6.1]{FH3}. Now, this corresponds to a negative prefundamental representations, and by \cite[Section 4.1]{HJ} we have
$$\chi_q(L(\Psib_{2,aq^{-3}}^{-1})) = \Psib_{2,aq^{-3}}^{-1} E_e(Q_{w(\omega_1),a}) = \Psib_{2,aq^{-3}}^{-1} \chi(L(\Psib_{2,aq^{-3}}^{-1})) E_e(Q_{w(\omega_1),a})(\bchi_{w(\omega_1)})^{-1}.$$
But by \cite[Corollary 4.11]{Hsh} this is also $\chi_q(L^{\on{sh}}(\Psib_{2,aq^{-3}}^{-1}))$.

\medskip

{\bf Type $B_2$} 

\medskip

We denote $d_1 = 2$, $d_2 = 1$. We have
$$\chi_{\omega_1}  = 1\text{ , } \chi_{2\omega_2 - \omega_1} = \frac{1}{1-a_1^{-1}}\text{ , }
 \chi_{\omega_1 - 2\omega_2}  = \frac{1}{(1- a_1^{-1}a_2^{-2})(1-a_2^{-1})},
$$
$$ \chi_{-\omega_2} = \chi_{1,s_1s_2s_1s_2} = \frac{1}{(1-a_1^{-1}a_2^{-2})(1-a_1^{-1})(1-a_1^{-1}a_2^{-1})},$$

$$\chi_{\omega_2}  = 1\text{ , } \chi_{\omega_1 - \omega_2}  = \frac{1}{1-a_2^{-1}}\text{ , }
 \chi_{\omega_2 - \omega_1}  = \frac{1}{(1 - a_1^{-1}a_2^{-1})(1 - a_1^{-1})^2},$$
$$
 \chi_{-\omega_2} =  \frac{1}{(1- a_1^{-1}a_2^{-1})(1 - a_2^{-1})(1 - a_1^{-1}a_2^{-2})^2}.
$$
The orbit of $\ell$-weights are : 
$$\Psib_{\omega_1,a} = \Psib_{1,a}\text{ , }\Psib_{2\omega_2-\omega_1} = \Psib_{1,aq^{-4}}^{-1}\Psib_{2,aq^{-2}}
\text{ , }\Psib_{\omega_1 - 2\omega_2} = \Psib_{2,aq^{-4}}^{-1}\Psib_{1,aq^{-2}}\text{ , }\Psib_{-\omega_1,a} = \Psib_{1,aq^{-6}}^{-1},$$
$$\Psib_{\omega_2,a} = \Psib_{2,a}\text{ , }\Psib_{\omega_1-\omega_2} = \Psib_{2,aq^{-2}}^{-1}\Psib_{1,aq^{-2}}\Psib_{1,a}
\text{ , }\Psib_{\omega_2 - \omega_1} = \Psib_{2,aq^{-4}}\Psib_{1,aq^{-6}}^{-1}\Psib_{1,aq^{-4}}^{-1}\text{ , }\Psib_{-\omega_2,a} = \Psib_{2,aq^{-6}}^{-1}.$$
The cases $w = e, s_2, s_1$ follow from the discussion above. It suffices to treat the cases :
$$\text{(A) }s_2s_1(\omega_1) = \omega_1 - 2 \omega_2\text{, (B) }(s_1s_2s_1)(\omega_1) = - \omega_1\text{, (C) }
(s_1s_2)(\omega_2) = \omega_2 - \omega_1\text{ , (D) }(s_2s_1s_2)(\omega_2) = - \omega_2.$$

The cases (B) and (D) correspond to negative prefundamental
representations and can be treated as in type $A_2$ above. More precisely, 
let $i = 1$ or $2$, we have $\Psib_{-\omega_i,a} = \Psib_{i,aq^{-6}}^{-1}$. By \cite[Section 4.7]{FH3} (formulas (4.34) and (4.47)) : 
$$E_e(Q_{-\omega_1,a}) = \sum_{0\leq \gamma\leq \beta/2 \leq \alpha} V_{1,aq^{-6}}^{(\alpha)}V_{2,aq^{-4}}^{(\beta)}V_{1,aq^{-4}}^{(\gamma)}.$$
$$E_e(Q_{-\omega_2,a}) = 
  \sum_{0\leq \beta\leq \alpha/2, 0\leq \beta'\leq (\alpha+1)/2, 0\leq
    \gamma\leq \text{Min}(1 + 2\beta, 2\beta'-1)}
  V_{2,aq^{-6}}^{(\alpha)}V_{1,aq^{-6}}^{(\beta)}V_{1,aq^{-4}}^{(\beta')}V_{2,aq^{-2}}^{(\gamma)}.$$
We have $\bchi_{-\omega_i} = \varpi(E_e(Q_{-\omega_i,a}))$. Now by \cite{HJ}, $\chi_q(L(\Psib_{i,aq^{-6}}^{-1}))$ can be obtained as a limit of 
$q$-characters of Kirillov-Reshetikhin modules whose $q$-character can be obtained algorithmically. Comparing these results, we obtain 
$$\chi_q(L(\Psib_{i,aq^{-6}}^{-1})) = \Psib_{i,aq^{-6}}^{-1} E_e(Q_{-\omega_i,a}) = \Psib_{i,aq^{-6}}^{-1} \chi(L(\Psib_{i,aq^{-6}}^{-1})) E_e(Q_{-\omega_i,a})(\bchi_{-\omega_i})^{-1}.$$
But by \cite[Corollary 4.11]{Hsh}, this is also
$\chi_q(L^{\on{sh}}(\Psib_{i,aq^{-6}}^{-1}))$.

Case (A): by formula (4.45) of \cite[Section 4.7]{FH3}, we have 
$$E_e(Q_{\omega_1 - 2\omega_2,a}) =  \sum_{0\leq 2\beta\leq \alpha} V_{2,aq^{-4}}^{(\alpha)}V_{1,aq^{-4}}^{(\beta)}.$$
An explicit construction of the representation $L^{\on{sh}}(\Psib_{1,aq^{-2}} \Psib_{2,aq^{-4}}^{-1})$ (similar to \cite[Example 5.2]{Hsh}) implies that its $q$-character is equal to 
$\Psib_{1,aq^{-2}} \Psib_{2,aq^{-4}}^{-1}E_e(Q_{\omega_1 - 2\omega_2,a})$.
We have again $\bchi_{s_2s_1(\omega_1)} = \varpi(E_e(Q_{\omega_1 - 2\omega_2,a}))$. 
Also, a computation as in \cite{FH} shows that
$$\chi_q(L(\Psib_{\omega_1 - 2 \omega_2,a})) = \Psib_{\omega_1 - 2 \omega_2,a} \chi(L(\Psib_{\omega_1 - 2 \omega_2,a})) E_e(Q_{\omega_1 - 2\omega_2,a})(\bchi_{s_2s_1(\omega_1)})^{-1}.$$

Case (C) : by \cite[Section 4.7]{FH3} Formula (4.41), we have
$$E_e(Q_{\omega_2 - \omega_1,a}) =  
 \sum_{0\leq \beta\leq
    \text{Min}(2\alpha,2\alpha'+1)}
  V_{1,aq^{-4}}^{(\alpha)}V_{1,aq^{-6}}^{(\alpha')}V_{2,aq^{-2}}^{(\beta)}.$$
An explicit construction of the representation $L^{\on{sh}}(\Psib_{2,aq^{-4}}\Psib_{1,aq^{-4}}^{-1}\Psib_{1,aq^{-6}}^{-1})$ 
(similar to \cite[Example 5.2]{Hsh}) gives that its $q$-character is equal to 
$\Psib_{2,aq^{-4}}\Psib_{1,aq^{-4}}^{-1}\Psib_{1,aq^{-6}}^{-1}E_e(Q_{\omega_2 - \omega_1,a})$.
Also, a computation as in \cite{FH} shows that
$$\chi_q(L(\Psib_{\omega_2 -\omega_1,a}) = \Psib_{\omega_2 - \omega_1,a} \chi(L(\Psib_{\omega_2 - \omega_1,a})) E_e(Q_{\omega_2 - \omega_1,a})(\bchi_{s_1s_2(\omega_2)})^{-1}.$$

\medskip

{\bf Type $G_2$} 

\medskip

We denote $d_1 = 3$, $d_2 = 1$. We have

$$\chi_{\omega_1}  = 1\text{ , } 
\chi_{3\omega_2 - \omega_1}  = \frac{1}{1-a_1^{-1}}\text{ , }
\chi_{2\omega_1 - 3\omega_2}  = \frac{1}{(1- a_1^{-1}a_2^{-3})(1-a_2^{-1})},
$$
$$ \chi_{-2\omega_1 + 3 \omega_2}  = \frac{1}{(1-a_1^{-1}a_2^{-3})(1-a_1^{-1})^2(1-a_1^{-1}a_2^{-1})},$$
$$ \chi_{\omega_1 - 3 \omega_2} =  \frac{1}{(1-a_1^{-2}a_2^{-3})(1-a_1^{-1}a_2^{-2})(1-a_1^{-1}a_2^{-3})^2(1-a_2)},$$
$$ \chi_{-\omega_1}  = \frac{1}{(1-a_1^{-2}a_2^{-3})^2(1-a_1^{-1}a_2^{-1})(1-a_1)(1-a_1^{-1}a_2^{-2})(1-a_1^{-1}a_2^{-3})},$$
$$\chi_{\omega_2}  = 1\text{ , } \chi_{\omega_1 - \omega_2}  = \frac{1}{1-a_2^{-1}}
\text{ , }
 \chi_{2\omega_2 - \omega_1}  = \frac{1}{(1 - a_1^{-1}a_2^{-1})(1 - a_1^{-1})^3},$$
$$
 \chi_{\omega_1 - 2\omega_2}  = \frac{1}{(1- a_1^{-1}a_2^{-1})(1 - a_2^{-1})^2(1 - a_1^{-1}a_2^{-2})},
$$
$$
 \chi_{\omega_2 - \omega_1} =  \frac{1}{(1- a_1^{-1}a_2^{-1})^2(1 - a_1^{-1})^3(1 - a_1^{-1}a_2^{-2})(1-a_1^{-2}a_2^{-3})^3},
$$
$$
 \chi_{-\omega_2}  = \frac{1}{(1- a_1^{-1}a_2^{-1})(1 - a_2^{-1})(1 - a_1^{-1}a_2^{-2})^3(1-a_1^{-2}a_2^{-3})^2(1-a_1^{-1}a_2^{-3})^3}.
$$
The orbit of $\ell$-weights are : 
$$\Psib_{\omega_1,a} = \Psib_{1,a}
\text{ , }\Psib_{3\omega_2-\omega_1} = \Psib_{1,aq^{-6}}^{-1}\Psib_{2,aq^{-3}}
\text{ , }\Psib_{2\omega_1 - 3\omega_2} = \Psib_{2,aq^{-5}}^{-1}\Psib_{1,aq^{-4}}\Psib_{1,aq^{-2}},$$
$$\Psib_{-2\omega_1 + 3\omega_2} = \Psib_{2,aq^{-7}}\Psib_{1,aq^{-10}}^{-1}\Psib_{1,aq^{-8}}^{-1}
\text{ , }\Psib_{\omega_1 - 3\omega_2} = \Psib_{2,aq^{-9}}^{-1}\Psib_{1,aq^{-6}}
\text{ , }\Psib_{-\omega_1,a} = \Psib_{1,aq^{-12}}^{-1},$$
$$\Psib_{\omega_2,a} = \Psib_{2,a}
\text{ , }\Psib_{\omega_1-\omega_2} = \Psib_{2,aq^{-2}}^{-1}\Psib_{1,aq}\Psib_{1,aq^{-1}}\Psib_{1,aq^{-3}},$$
$$\Psib_{2\omega_2-\omega_1} = \Psib_{2,aq^{-4}}\Psib_{2,aq^{-6}}\Psib_{1,aq^{-5}}^{-1}\Psib_{1,aq^{-7}}^{-1}\Psib_{1,aq^{-9}}^{-1},
\Psib_{\omega_1-2\omega_2} = \Psib_{2,aq^{-6}}^{-1}\Psib_{2,aq^{-8}}^{-1}\Psib_{1,aq^{-3}}\Psib_{1,aq^{-5}}\Psib_{1,aq^{-7}},$$
$$\Psib_{\omega_2-\omega_1} =  \Psib_{2,aq^{-10}}\Psib_{1,aq^{-9}}^{-1}\Psib_{1,aq^{-11}}^{-1}\Psib_{1,aq^{-13}}^{-1}
\text{ , }\Psib_{-\omega_2,a} = \Psib_{2,aq^{-12}}^{-1}.$$

The results for the weight $\omega_i$, $s_i(\omega_i)$, $i = 1,2$, have been treated above. The weights $-\omega_i$, $i = 1,2$, 
correspond to negative prefundamental representations and can be treated as in type $A_2$ above. The explicit formula for 
$Q_{i,a}$ can be found in \cite[Section 4.8]{FH3}.

The other intermediate weights are treated as for type $B_2$ above. For example, we have
$$E_e(Q_{2\omega_1-3\omega_2,a}) = \sum_{0\leq 3 \beta\leq\alpha} V_{2,aq^{-5}}^{(\alpha)}V_{1,aq^{-6}}^{(\beta)},$$
and
$$\chi_q(L(\Psib_{2\omega_1 - 3 \omega_2,a})) = \Psib_{2\omega_1 - 3 \omega_2,a} \chi(L(\Psib_{2\omega_1 - 3 \omega_2,a})) E_e(Q_{2\omega_1 - 3\omega_2,a})
(\bchi_{s_2s_1(\omega_1)})^{-1}.$$

\end{document}